\documentclass[12pt,reqno]{amsart}

\usepackage{a4wide}
\usepackage{hyperref}
\usepackage[inline]{enumitem}

\usepackage{tikz}
\usetikzlibrary{decorations.pathreplacing, tikzmark}
\usetikzlibrary{positioning}

\newtheorem{theorem}[subsection]{Theorem}
\newtheorem{lemma}[subsection]{Lemma}
\newtheorem{proposition}[subsection]{Proposition}

\providecommand{\Z}{\mathbb{Z}}
\providecommand{\N}{\mathbb{N}}
\providecommand{\C}{\mathbb{C}}
\providecommand{\R}{\mathbb{R}}
\providecommand{\Q}{\mathbb{Q}}
\providecommand{\F}{\mathbb{F}}
\providecommand{\E}{\mathop{\mathbb{E}}}
\renewcommand{\P}{\mathop{\mathbb{P}}}
\providecommand{\wh}{\widehat}
\providecommand{\wt}{\widetilde}
\newcommand{\dd}{\,\mathrm{d}}
\providecommand{\Span}{\mathop{\rm Span}\nolimits}
\providecommand{\supp}{\mathop{\rm supp}\nolimits}
\renewcommand{\Re}{\mathop{\rm Re}\nolimits}

\numberwithin{equation}{section}

\newcounter{constcntbig} 
\newcounter{constcntlittle} 

\makeatletter
\newcommand{\newconstbig}[1]{%
    \refstepcounter{constcntbig}%
    \hypertarget{const:#1}{\mbox{}}%
    \protected@write\@auxout{}%
        {\string\newlabel{const:#1}{{\arabic{constcntbig}}{\thepage}}}%
    C_{\arabic{constcntbig}}%
}
\makeatother
\makeatletter
\newcommand{\newconstlittle}[1]{%
    \refstepcounter{constcntlittle}%
    \hypertarget{const:#1}{\mbox{}}%
    \protected@write\@auxout{}%
        {\string\newlabel{const:#1}{{\arabic{constcntlittle}}{\thepage}}}%
    c_{\arabic{constcntlittle}}%
}
\makeatother

\newcommand{\refconstlittle}[1]{%
    \hyperlink{const:#1}{c_{\ref{const:#1}}}%
}
\newcommand{\refconstbig}[1]{%
    \hyperlink{const:#1}{C_{\ref{const:#1}}}%
}

\begin{document}

\title{On Rado's single equation theorem}

\author{Tom Sanders}
\address{Mathematical Institute\\
University of Oxford\\
Radcliffe Observatory Quarter\\
Woodstock Road\\
Oxford OX2 6GG\\
United Kingdom}
\email{tom.sanders@maths.ox.ac.uk}

\begin{abstract}
We show that for non-zero integers $a$ and $b$ there is a natural number $N \leq \exp(r^{2+o_{a,b;r\rightarrow \infty}(1)})$ such that in any $r$-colouring of $\{1,\dots,N\}$ there are $x,y,z$, all in the same colour class, such that $ax-ay=bz$. 
\end{abstract}

\maketitle

\section{Introduction}\label{sec.intro}

For us an $r$-colouring of a set $X$ is a cover $\mathcal{C}$ of $X$ of size $r$.  For a $k \times d$ matrix $M$ with integer entries we define $\mathfrak{R}_M:\N \rightarrow \N \cup \{\infty\}$ to be the map taking $r \in \N$ to the smallest natural $N$ such that in any $r$-colouring $\mathcal{C}$ of $[N]:=\{1,\dots,N\}$ there is some $A \in \mathcal{C}$ and $x \in A^d$ such that $Mx=0$ (and $\mathfrak{R}_M(r):=\infty$ if no such $N$ exists).

Following \cite[p425]{rad::1} we say the matrix $M$ is partition regular if $\mathfrak{R}_M(r)<\infty$ for all $r \in \N$. Rado's Theorem \cite[Satz IV, p445]{rad::1} characterises when $M$ is partition regular. In the case where $M$ is a single row, say $M=a=(a_1,\dots,a_d)$, it reduces \cite[p446]{rad::1} to saying that $a$ is partition regular if and only if there is some $I \subset [d]$ with $a_j \neq 0$ for some $j \in I$, and $\sum_{i \in I}{a_i}=0$.

In this paper our interest is in the growth of $\mathfrak{R}_a(r)$, as a function of $r$, when $a\in \Z^d$ is partition regular. In \cite[Theorem 1.5]{cwasch::0} Cwalina and Schoen show that for any partition regular $a$ there is $C_a\leq 4$ such that 
\begin{equation}\label{eqn.upper}
\mathfrak{R}_a(r) \leq \exp(r^{C_a+o_{a;r \rightarrow \infty}(1)}).
\end{equation}
There are various refinements in \cite{cwasch::0} and \cite{Kosciuszko:2025aa} for particular classes of $a$ which we shall discuss later, but in fact it has been known since work of Schur \cite[p114, Hilfssatz]{sch::4} that $C_{(1,-1,-1)}\leq 1$.  On the other hand, Fox has shown (recorded in \cite[Theorem 5]{foxkle::}) that provided $a$ is not invariant, we cannot have $C_a<1$. ($a$ is said to be invariant if $\sum_{i=1}^d{a_i}=0$, and in this case $\mathfrak{R}_a(r)=1$ for all $r$.) The main purpose of this note is to prove that we may take $C_a \leq 2$ for all partition regular $a$:
\begin{theorem}\label{thm.mn}
Suppose that $a \in \Z^d$ is partition regular.  Then
\begin{equation*}
\mathfrak{R}_a(r) = \exp(O_a(r^2\log^{O(1)}2r))
\end{equation*}
\end{theorem}
In \cite[Theorem 1.7]{cwasch::0} Cwalina and Schoen show that one may take $C_a \leq 2$ for a wide class of $a$, and \cite[Theorem 1.6]{cwasch::0} shows one may take $C_a \leq 3$ for all $a$ involving at least four variables. On the other hand the recent breakthrough work of Kelley and Meka \cite{kelmek::0} provides a way of lifting counts for equations with three variables into counts for equations with four variables and so it is natural to expect this to give an improvement to Cwalina and Schoen's results. 

In \S\ref{sec.model} we shall illustrate our argument for a toy version of the problem. This is essentially just the sparsity-expansion dichotomy of Chapman and Prendiville described in \cite[Lemma 2.3]{chapre::} combined with the quantitative refinements of Kelley and Meka in \cite[Lemma 1.5]{kelmek::0}. No difficulties arise in combining them.

In \S\ref{sec.bourganised}, we use the Bohr set machinery of Bourgain \cite{bou::5} to translate the above argument into a proof of Theorem \ref{thm.mn}. Bloom and Sisask \cite{Bloom:2023aa} have already extended Kelley and Meka's work to Bohr sets, and we encounter no new difficulties here. The fact that Theorem \ref{thm.mn} gives $C_a\leq 2$ rather than $C_a\leq 1$ looks like a typical artefact of this machinery: the basic argument is iterative and takes $r^{1+o(1)}$ steps and at each step we have to pay the price of an additional dilation. 

\subsection{When can we do better?} The aforementioned fact that $C_{(1,-1,-1)}=1$ can be proved by first showing that $\mathfrak{R}_{(1,-1,-1)}(r)+2 \leq R(3,\dots,3)$, where the right hand side is the smallest natural $n$ such that any $r$-colouring of the complete graph on $n$ vertices contains a monochromatic triangle. A bound on this Ramsey number such as \cite[Corollary 3]{gregle::0} can then be used to complete the argument.

The first part of this argument has a generalisation in \cite[(3.4), p181]{Abbott:1972aa} which shows that $\mathfrak{R}_M(r) +2 \leq R(k,\dots,k)$ where $M$ is the $\binom{k-1}{2} \times \binom{k}{2}$ integer-valued matrix encoding\footnote{Meaning, if the rows of $M$ are indexed by pairs $(i,j)$ with $1 \leq i < j \leq k-1$ and the columns are indexed by pairs $(i,j)$ with $1 \leq i < j \leq k$, then the $(i,j)$th row of $M$ has a $1$ in columns $(i,j)$ and $(j,j+1)$, a $-1$ in column $(i,j+1)$, and $0$s elsewhere.} the system of equations
\begin{equation}\label{eqn.system}
x_{i,j}+x_{j,j+1}=x_{i,j+1} \text{ whenever }1 \leq i < j \leq k-1.
\end{equation}
The same \cite[Corollary 3]{gregle::0} implies that $R(k,\dots,k)\leq \exp(r^{1+o_k(1)})$, and since $\mathfrak{R}_a(r) \leq \mathfrak{R}_M(r)$ for any $a$ in the $\Q$-row span of $M$ we conclude that $C_a=1$ for any $a$ in the $\Q$-row span of $M$. In particular, if the $x_{i,j}$s satisfy (\ref{eqn.system}) then 
\begin{equation*}
x_{1,2}+x_{2,3}+\cdots + x_{k-1,k}=x_{1,k},
\end{equation*}
 so $C_{(1,\dots,1,-1)}=1$ where $1$ occurs $k-1$ times in the row. Similarly, 
 \begin{equation*}
 x_{1,k-1}+\cdots + x_{k-2,k-1} + (k-1)x_{k-1,k} = x_{1,k}+\cdots + x_{k-2,k},
 \end{equation*}
  so $C_{(1,\dots,1,k-1,-1,\dots,-1)}=1$ where $1$ and $-1$ both occur $k-1$ times in the row. These results are due to Ko{\'s}ciuszko \cite[Theorems 6 \& 7]{Kosciuszko:2025aa}.
  
 One can make more vectors $a$ with $C_a=1$ from existing ones by appending a vector whose entries sum to $0$, so that from $C_{(1,-1,-1)}=1$ we deduce $C_{(1,\dots,1,-1,\dots,-1)}=1$ where $1$ occurs $k$ times and $-1$ occurs $k+1$ times. This is the monotonicity of $k$-sum-free sets discussed on \cite[p604]{cwasch::0}. Combining this with our earlier observations, for $a,b\in \N$ we have $C_{(a,-a,b,1,\dots,1,-1,\dots,-1)}=1$ if the $1$s and $-1$s each occur $b$ times. By way of comparison, Theorem \ref{thm.mn} comes down to showing $C_{(a,-a,b)}\leq 2$.

\section{The toy}\label{sec.model}

Toy problems over finite fields have been standard expository devices since the paper of Green \cite{gre::9} (updated in the sequels \cite{wol::3} \& \cite{Peluse:2024aa}). For our toy we replace $\{1,\dots,N\}$ by the vector space $G:=\F_q^n$, where $\F_q$ is a finite field with $q$ elements. We fix $G$ in this way for the remainder of this section only. We shall show the following:
\begin{theorem}\label{thm.model}
Suppose that $a,b \in \F_q^*$ and $\mathcal{C}$ is an $r$-colouring of $G$. Then there is a colour class $A \in \mathcal{C}$ and $\exp(-O_q(r\log^92r)) |G|^2$ triples $(x,y,z) \in A^3$ such that $ax-ay = bz$.
\end{theorem}
This toy is supposed to illustrate the argument we will eventually use for the special case of Theorem \ref{thm.mn} where $a=(a,-a,-b)$ with $a,b \in \Z^*$, which will itself imply the more general case.

We need Proposition \ref{prop.sum} below, and to record this we need some notation. For $V \leq G$ we write $\mu_V$ for the uniform probability measure on $V$ extended to be a probability measure on $G$, \emph{i.e.\ }$\mu_V(A)=|A \cap V|/|V|$ for all $A \subset G$. We also write $\ast$ for convolution on $G$ so that for functions $f$ and $g$ on $G$, and a measure $\mu$ on $G$ (defined on all subsets of $G$) we have
\begin{equation}\label{eqn.earlier}
f \ast g(x):=\int{f(x-y)g(y)\dd\mu_G(y)}, \text{ and } f \ast \mu(x):=\int{f(x-y)\dd\mu(y)},
\end{equation}
and
\begin{equation}\label{eqn.inner}
\langle f,g\rangle_{L_2(\mu)}:=\int{f(x)\overline{g(x)}\dd\mu(x)}.
\end{equation}

\begin{proposition}\label{prop.sum} Suppose that $V \leq G$, $A,D \subset V$ are non-empty with $\alpha:=\mu_V(A)$ and $\delta:=\mu_V(D)$, and $\epsilon \in (0,1]$ is a parameter. Then either
\begin{equation*}
|\langle 1_{-A}  \ast  1_{A}, 1_D\rangle_{L_2(\mu_G)}-\alpha^2\delta\mu_G(V)^2| \leq \epsilon \alpha^2\delta\mu_G(V)^2; 
\end{equation*}
 or else there is $V' \leq V$ with $\dim V/V' = O_\epsilon((\log^4 2\alpha^{-1})(\log^4 2\delta^{-1}))$ such that $\|1_A \ast \mu_{V'}\|_\infty \geq (1+\Omega(\epsilon))\alpha$.
\end{proposition}
This is a minor variant of \cite[Proposition 13, p32]{Bloom:2023aa} with $1_A \ast 1_A$ replaced by $1_{-A} \ast 1_{A}$ and rescaled to be stated for subgroups of a larger ambient group which is helpful for our application. 

\begin{proof}[Proof of Theorem \ref{thm.model}]
Write $\mathcal{C}:=\{A_1,\dots,A_r\}$. We proceed iteratively to construct $G=:V_0 \geq V_1\geq \dots$. Writing $\cdot$ for the scalar multiplication of $\F_q$ on $G$, set
\begin{equation*}
S_{i,j}:=\|1_{a\cdot A_j}\ast \mu_{V_i}\|_\infty \text{ for all } j \in [r].
\end{equation*}
Since $V_i \leq V_{k}$ whenever $i \geq k$ we have $\mu_{V_i} \ast \mu_{V_k} = \mu_{V_k}$ and hence
\begin{equation}\label{mon}
S_{i,j} \geq S_{k,j}\text{ whenever }i \geq k, \text{ and } S_{i,j} \leq 1 \text{ for all }i.
\end{equation}

At step $i$ of the iteration we shall show that either
\begin{enumerate}
\item\label{c1} there is some $j \in [r]$ such that
\begin{equation*}
\langle 1_{a\cdot A_j} \ast 1_{-a\cdot A_j},1_{b \cdot A_j}\rangle_{L_2(\mu_G)} \geq \frac{1}{2r^3}\mu_G(V_i)^2;\text{ or }
\end{equation*}
\item\label{c2} there is $V_{i+1} \leq V_i$ and $j \in [r]$ such that
\begin{enumerate}
\item\label{c21} $\dim V_i/V_{i+1} =O(\log^82r)$;
\item\label{c22} $S_{i+1,j} \geq (1+\Omega(1))S_{i,j}$; and
\item\label{c23} $S_{i,j} \geq 1/r$.
\end{enumerate}
\end{enumerate}
We stop the iteration the first time we are in case (\ref{c1}). Suppose that we are in case (\ref{c2}) for a particular $j$ at steps $i_1 <i_2<\dots < i_k$ of the iteration. Then by (\ref{mon}) and (\ref{c22}) \& (\ref{c23}) we have
\begin{align*}
1 \geq S_{i_k+1,j} \geq (1+\Omega(1))S_{i_k,j} & \geq (1+\Omega(1))S_{i_{k-1}+1,j}\\ & \geq \qquad \dots \qquad  \geq (1+\Omega(1))^kS_{i_1,j} \geq (1+\Omega(1))^k\cdot (1/r).
\end{align*}
It follows that $k=O(\log 2r)$. Since there are $r$ possible values for $j$ we conclude that we must have terminated the iteration at step $i_0=r\cdot O(\log 2r)$, in which case
\begin{equation*}
\dim G/V_{i_0} =\dim V_0/V_1 + \cdots + \dim V_{i_0-1}/V_{i_0} =i_0\cdot O(\log^82r) = O(r\log^92r)
\end{equation*} by (\ref{c21}).

Since $\mu_G(V_{i_0})=q^{-\dim G/V_{i_0}}$ and $|G|^2\langle 1_{a\cdot A_j} \ast 1_{-a\cdot A_j},1_{b \cdot A_j}\rangle_{L_2(\mu_G)} $ is exactly the number of solutions to $ax-ay=bz$ for $x,y,z \in A_j$ we have the conclusion of Theorem \ref{thm.model} from (\ref{c1}) as claimed.

It remains to show that at each stage of the iteration we are either in case (\ref{c1}) or (\ref{c2}). Suppose we are at step $i$ of the iteration. Since $b \in \F_q^*$ we have $b\cdot V_i=V_i$ and $\mathcal{C}$ is a cover of $G$, and there is some $j=j(i) \in [r]$ such that $\mu_{V_i}(b\cdot A_j) \geq \frac{1}{r}$. Since $a,b \in \F_q^*$ we have $(ab^{-1})V_i=V_i$, and
\begin{equation*}
S_{i,j}=\|1_{a\cdot A_j}\ast \mu_{V_i}\|_\infty \geq \mu_{V_i}(a\cdot A_j) = \mu_{V_i}(b \cdot A_j).
\end{equation*}
In particular $S_{i,j} \geq 1/r$. Let $x_i\in G$ be such that 
\begin{equation*}
\mu_{V_i}(x_i-a \cdot A_j )=1_{a \cdot A_j} \ast \mu_{V_i}(x_i)=\|1_{a\cdot A_j}\ast \mu_{V_i}\|_\infty.
\end{equation*}
Apply Proposition \ref{prop.sum} to $V_i\leq G$ with $\epsilon:=\frac{1}{2}$, and $A:=(x_i-a \cdot A_j)\cap V_i$ and $D:=(b \cdot A_j)\cap V_i$ to get that either
\begin{align*}
\langle 1_{a \cdot A_j} \ast 1_{-a\cdot A_j}, 1_{b \cdot A_j}\rangle_{L_2(\mu_G)} & =  \langle 1_{-(x_i-a \cdot A_j)} \ast 1_{x_i-a\cdot A_j}, 1_{b \cdot A_j}\rangle_{L_2(\mu_G)}\\ &\geq \langle 1_{-A} \ast 1_{A},1_D\rangle_{L_2(\mu_{G})}\\ &\geq \frac{1}{2}\mu_{V_i}(A)^2\mu_{V_i}(D)\mu_G(V_i)^2 \geq \frac{1}{2r^3}\mu_G(V_i)^2;
\end{align*}
or else there is $V_{i+1} \leq V_i$ with $\dim V_i/V_{i+1}=O(\log^82r)$ and $S_{i+1,j} \geq (1+\Omega(1))S_{i,j}$.  In the first case we are in (\ref{c1}); in the second we are in (\ref{c2}). The result is proved.
\end{proof}

\section{Proof of Theorem \ref{thm.mn}}\label{sec.bourganised}

We shall show the following result.
\begin{theorem}\label{thm.keycount}
There is a function $f:\Z^*\times \Z^* \times \N^* \rightarrow (0,1]$ such that if $\mathcal{C}$ is an $r$-colouring of $\{-N,\dots,N\}$, then there is a colour class $A \in \mathcal{C}$ and at least $f(a,b,r)N^2$ triples $(x,y,z) \in A^3$ such that $ax-ay = bz$. Moreover, $f(a,b,r)=\exp(-O_{a,b}(r^2\log^{O(1)} 2r))$.
\end{theorem}

\begin{proof}[Proof of Theorem \ref{thm.mn} assuming Theorem \ref{thm.keycount}]
Let $f$ be the function in Theorem \ref{thm.keycount}.  By Rado's Single Equation Theorem \cite[Theorem 9.5, p255]{lanrob::} there is $I \subset [d]$ such that $a_j \neq 0$ for some $j \in I$ and $\sum_{i\in I}{a_i}=0$. Let $b:=-\sum_{i \not \in I}{a_i}$, which we may assume is nonzero since otherwise $a$ is invariant and $\mathfrak{R}_a(r)=1$ for all $r$, and finally let $N$ be such that $f(a_j,b,2r+1)N^2>1$, so that $N=\exp(O_{a}(r^2\log^{O(1)}2r))$.

Given an $r$-colouring $\mathcal{C}$ of $\{1,\dots,N\}$, $\mathcal{C}':=\{A,-A : A \in \mathcal{C}\}\cup \{\{0\}\}$ is a $(2r+1)$-colouring of $\{-N,\dots,N\}$. By Theorem \ref{thm.keycount} there is a colour class $A' \in \mathcal{C}'$ with $f(a_j,b,2r+1)N^2$ triples $(y,z,w) \in A^3$ such that $a_jy-a_jz=bw$. Since $f(a_j,b,2r+1)N^2>1$ we see that $A' \neq \{0\}$, and hence $A' \in \mathcal{C}$ or $-A' \in \mathcal{C}$. In the former case let $x \in [N]^d$ be defined by $x_j= y$; $x_k= z$ for all $k \in I \setminus \{j\}$; and $x_k= w$ for all $k \not \in I$. Then $x \in (A')^d$ and
\begin{equation*}
a_1x_1+\cdots + a_dx_d = a_jy + \left(\sum_{i \in I\setminus \{j\}}{a_i}\right)z + \left(\sum_{i \not \in I}{a_i}\right)w = a_jy-a_jz-bw=0.
\end{equation*}
The case for $-A' \in \mathcal{C}$ is similar, we just replace $y$, $z$ and $w$ by $-y$, $-z$ and $-w$. In particular, in either case we have some $A \in \mathcal{C}$ and $x \in A^d$ such that $a_1x_1+\cdots + a_dx_d=0$ and hence $\mathfrak{R}_a(r) \leq N$. The result is proved.
\end{proof}

\subsection{Notation} First we set some notation extending the notation from \S\ref{sec.model}. We shall work in a finite Abelian group $G$, and we only consider measures on $G$ defined on all subsets of $G$. For $\emptyset \neq A\subset G$ we write $\mu_A$ for the uniform probability measure supported on $A$. If $\mu$ is a measure and $f \in L_\infty(\mu_G)$ then $f\dd \mu$ denotes the measure assigning mass $\int{f1_E\dd\mu}$ to the set $E \subset G$. 

For functions $f$ and $g$ and a measure $\mu$ on $G$ we define convolution and the inner product as in (\ref{eqn.earlier}) and (\ref{eqn.inner}). If $\nu$ is a further measure on $G$ then we define $\nu \ast \mu$ to be the measure given by
\begin{equation*}
\nu\ast \mu(E):=\int{1_E(x+y)\dd\nu(x)\dd\mu(y)} \text{ for all }E \subset G.
\end{equation*}
For $\mu$ a measure on $G$ we write $\langle f,\mu\rangle:=\int{f(x)\dd\mu(x)}$, and $\wt{\mu}(E):=\overline{\mu(-E)}$ for all $E \subset G$.

We write $\wh{G}$ for the dual group of $G$, that is the set of homomorphisms $G \rightarrow S^1$, and define the Fourier and Fourier-Stieltjes transform of a function $f$ and a measure $\mu$ respectively by
\begin{equation*}
\wh{f}(\gamma):=\int{f(x)\overline{\gamma(x)}\dd\mu_G(x)} \text{ and } \wh{\mu}(\gamma):=\int{\overline{\gamma(x)}\dd\mu(x)}.
\end{equation*}

The labels $C_1,C_2,\dots \geq 1$ and $c_1,c_2,\dots \in (0,1]$ denote absolute constants which remain the same in all instances. The $C_i$s are different to the $C_a$s defined in \S\ref{sec.intro}.

For $\Gamma$ a set of characters on $G$ and $\delta \in (0,1]$, the Bohr set with frequency set $\Gamma$ and width $\delta$ is
\begin{equation*}
B(\Gamma,\delta):=\{x \in G: |\gamma(x)-1| \leq \delta \text{ for all }\gamma \in \Gamma\}.
\end{equation*}
There are two standard ways to define Bohr sets. The other way is \cite[\S4.4, p187]{taovu::}, and for most purposes they are equivalent via \cite[(4.25), p187]{taovu::}. 

\subsection{The main iteration lemma}\label{sec.mil} We shall prove Theorem \ref{thm.keycount} by copying the proof of Theorem \ref{thm.model}. We replace subspaces with Bohr sets, and codimension with the size of the frequency set, and pay an additional cost in the width parameter. 

It can be helpful to think of Bohr sets as somewhat like convex bodies.  If $B_0$ is a convex body in $\R^d$ and $B_1$ is that same body dilated by a factor $\delta$, then $\mu_{\R^d}(B_1+B_0)=(1+O(d\delta))\mu_{\R^d}(B_0)\approx \mu_{\R^d}(B_0)$ (here $\mu_{\R^d}$ is Lebesgue measure on $\R^d$). In Lemma \ref{lem.regular} we shall see that many Bohr sets share this property.

In general for any sets $B_0,B_1 \subset G$ having $\mu_G(B_1+B_0) \approx \mu_G(B_0)$ we have $x+B_0 \approx B_0$ for all $x \in B_1-B_1$. We think of this as saying that $B_1-B_1$ acts approximately on $B_0$, and if this approximation is good enough then we can transfer arguments that work for genuine actions to this approximate setting. 

To copy the proof of Theorem \ref{thm.model} we need a generalisation of Proposition \ref{prop.sum}. This is Lemma \ref{lem.itlem} below, but it is complicated enough that we cannot make do with just one pair $B_0$ and $B_1$ with $\mu_G(B_1+B_0)\approx \mu_G(B_0)$, and instead will need a whole chain of sets $B_0,B_1,\dots$.

We can recover Proposition \ref{prop.sum} in the case $\epsilon=1/2$ --  the case we actually used -- from Lemma \ref{lem.itlem} by setting $G=\F_q^n$,  and $B_0=B_1=B_2=B_3=B_4=B_5=V \leq G$ and $k$, $m$, and $l$ minimal satisfying the hypotheses. This follows because a Bohr set contains the annihilator of its frequency set which, if $G=\F_q^n$, is a subspace of codimension at most the size of that frequency set -- in this case at most $m$.

\begin{lemma}[Iteration lemma]\label{lem.itlem} Suppose that $\alpha,\delta \in (0,1]$,  $k$, $m$, and $l$ are naturals with $k \geq \newconstbig{spec}\log 2\delta^{-1}$, $\tikzmark{b1} l\geq   \newconstbig{spec3}k \log 2\alpha^{-1}$, and $ \tikzmark{b2} m \geq \newconstbig{spec2}l^2k^2\log^22\alpha^{-1}$,
\begin{tikzpicture}[overlay, remember picture, >=stealth]

\node[font=\tiny, left=3cm of pic cs:b1, yshift=-0.5\baselineskip, align=left, red] (explain)
{ parameters for\\  how many times \\ we `need to add'\\ approximate\\ actions};

\draw[->, red,rounded corners=5pt]
  ([xshift=0.5cm]explain.east) -| (pic cs:b1);

\draw[->, red, rounded corners=5pt]
  ([xshift=0.5cm]explain.east) -| (pic cs:b2);
\end{tikzpicture}
\[
\begin{array}{l l}
\tikzmark{start}  \mathord{\bullet}\,\mu_G(-B_1-B_2+{B_0}) \leq \left(1+\newconstlittle{me}\alpha\right)\mu_G({B_0}) & \bullet \mu_G(B_2+B_1+{B_0}) \leq \left(1+\newconstlittle{16}\alpha^2\right)\mu_G({B_0})\\[5pt]
\mathord{\bullet}\,\mu_G(B_1-B_2+l(B_3-B_3))\leq 2\mu_G(B_1) & \bullet \mu_G(B_1+B_3) \leq 2\mu_G(B_1)\\[5pt]
\mathord{\bullet}\,\mu_G(B_2+l(B_3-B_3))\leq 2\mu_G(B_2) &\bullet \mu_G(B_3+mB_4) \leq 2\mu_G(B_3)\\[5pt]
\tikzmark{end}\mathord{\bullet}\,\mu_G(B_5+B_4)\leq \left(1+\newconstlittle{96}\alpha^{4k}\right)\mu_G(B_4) & 
\end{array}
\]
\begin{tikzpicture}[overlay, remember picture]
\draw[decorate, red, decoration={brace,amplitude=10pt}]
  ([xshift=\textwidth-0.5cm, yshift=\baselineskip]pic cs:start) -- ([xshift=\textwidth-0.5cm ]pic cs:end)
  node[midway, right=10pt, align=left, text width=3cm] 
{\tiny conditions \\ controlling the\\ error in our\\ approximate\\ actions\\};
\end{tikzpicture}
\begin{tikzpicture}[overlay, remember picture, >=stealth]

\node[font=\tiny, left=5.8cm of pic cs:b3, yshift=\baselineskip, align=left, red] (explain2)
{defined s.t.\ \\  $B_0 +\supp \mu \approx B_0$,\\ and $\wh{\mu}\geq 0$ for\\ spectral positivity;\\ see Appendix \ref{ap.specpos}};

\draw[->, red, rounded corners=5pt]
  (explain2.east) -| ([xshift=4pt, yshift=6pt]pic cs:b3);
\end{tikzpicture}
the set $A$ has $\mu_{B_0}(A)=\alpha$, and \tikzmark{b3} $  \tikzmark{b3} \mu :=   \wt{\mu_{B_1}}\ast \wt{\mu_{B_2}} \ast \mu_{B_2} \ast \mu_{B_1}$ has $\mu(D)\geq \delta$. Then, either
\begin{equation*}\setcounter{enumi}{2}
\langle 1_A \ast 1_{-A}, 1_{D}\rangle_{L_2(\mu)} \geq \frac{1}{2}\delta\alpha^2\mu_G({B_0});
\end{equation*}
or there is a Bohr set $B_6$ with frequency set of size at most $m$ and width $\newconstlittle{128}\frac{\alpha^{4k}}{m}$ such that
\begin{equation*}
\|1_{A \cap B_0} \ast \nu\|_\infty \geq (1+\newconstlittle{32})\alpha \text{ for all probability measures }\nu\text{ supported on }B_6 \cap (B_5 - B_5).\tikzmark{b4} 
\end{equation*}
\begin{tikzpicture}[overlay, remember picture, >=stealth]

\node[font=\tiny, right=0.5cm of pic cs:b4, align=left, red] (explain3)
{`hereditary'\\ density increment\\ with better\\ monotonicity;\\ c.f.\ Lemma \ref{lem.large}};

\draw[->, red]
  ([xshift=0cm, yshift=5pt]explain3.west) -- ([yshift=5pt]pic cs:b4);
\end{tikzpicture}
\end{lemma}

The proof we give here is a combination of spectral positivity (Appendix \ref{ap.specpos}), sifting (Appendix \ref{ap.sift}), and Proposition \ref{prop.d}, the proof of which appears immediately after this. These are the same arguments as in the work of Kelley and Meka \cite{kelmek::0}, as adapted to Bohr sets by Bloom and Sisask \cite{Bloom:2023aa}. It is more or less a combination of \cite[Propositions 15, 18, \& 20]{Bloom:2023aa} with some minor cosmetic and technical differences.
\begin{proof}
Let $\refconstbig{spec}>1$ be such that if $k \geq \refconstbig{spec}\log 2\delta^{-1}$ then $\left(1+\frac{1}{2}\right)(\delta/2)^{1/2k} \geq 1+3/8$ for all $\delta \in (0,1]$ and $k \geq 4\refconstbig{rdc}\log 2\cdot 32$. Let $\refconstbig{spec3}>1$ be such that if $l \geq \refconstbig{spec3}k \log 2\alpha^{-1}$, then $l \geq \refconstbig{L}\log  (2 \cdot 32 \cdot \alpha^{-4k})$ for all $\alpha \in (0,1]$ and $k \in \N^*$.  Let $\refconstbig{spec2}>1$ be such that if $m \geq \refconstbig{spec2}l^2k^2\log^22\alpha^{-1}$ then $m \geq \refconstbig{pd}\cdot (1/32)^{-2}l^2\cdot \log^4 2\alpha^{-4k}$ for all $k,l\in \N^*$, $\alpha \in (0,1]$. Let $\refconstlittle{me}=\frac{1}{8\refconstbig{specpos}}$, $\refconstlittle{16}=\frac{1}{16}$, $\refconstlittle{96}=\frac{1}{32}\refconstlittle{962}$, $\refconstlittle{128}=\frac{1}{32}\refconstlittle{8}$, and $\refconstlittle{32}=\frac{1}{32}$.

Apply spectral positivity -- Lemma \ref{lem.specpos} with the Lemma's $k$ equal to $k$; the Lemma's $\epsilon$ equal to $\frac{1}{2}$; the Lemma's $\eta$ equal to $\frac{\alpha}{8\refconstbig{specpos}}$; the Lemma's $\alpha$ equal to $\alpha$; the Lemma's $\delta$ equal to $\delta$; the Lemma's ${B_0}$ equal to ${B_0}$, the Lemma's $B_1$ equal to $-B_1-B_2$; the Lemma's $A$ equal to $A$; the Lemma's $\mu$ equal to $\wt{\mu_{B_1}} \ast \wt{\mu_{B_2}} \ast \mu_{B_2} \ast \mu_{B_1}$; and the Lemma's $D$ equal to $D$. Either we are done because
\begin{equation*}
\langle 1_A \ast 1_{-A}, 1_{D}\rangle_{L_2(\mu)} \geq \frac{1}{2}\delta\alpha^2\mu_G({B_0});
\end{equation*}
or else 
\begin{equation*}
\|1_{A \cap {B_0}} \ast 1_{-(A\cap {B_0})}\|_{L_{2k}(\mu)} \geq \left(1+\frac{1}{4}\right)\alpha^2\mu_G({B_0}).
\end{equation*}
Apply sifting -- Lemma \ref{lem.drc} with the Lemma's $\alpha$ equal to $\alpha$; the Lemma's $\epsilon$ equal to $\frac{1}{4}$; the Lemma's $\kappa$ equal to $\frac{1}{32}$; the Lemma's $k$ equal to $k$; the Lemma's ${B_0}$ equal to ${B_0}$; the Lemma's $B_1$ equal to $B_1$; and the Lemma's $B_2$ equal to $B_2$. This gives us sets $S$ and $T$ and elements $z$ and $w$ with
\begin{equation*}
\mu_{B_2+z}(T) ,\mu_{-B_1-w}(S) \geq \alpha^{4k} \text{ and }\langle 1_D,\mu_{T} \ast \mu_{S}\rangle \geq 1-\frac{1}{32}
\end{equation*}
where
\begin{equation*}
D=\left\{u \in -B_1-B_2+B_2+B_1: 1_{A\cap {B_0}} \ast 1_{-(A\cap {B_0})}(u)>\left(1+\frac{1}{8}\right)\alpha^2\mu_G({B_0})\right\}.
\end{equation*}
Apply Proposition \ref{prop.d} (telling us that sumsets correlate with Bohr sets of low rank) with the Proposition's $\epsilon$ equal to $\frac{1}{32}$; the Proposition's $\sigma$ and $\tau$ both equal to $\alpha^{4k}$; the Proposition's $l$ equal to $l$, which satisfies the required inequality; the Proposition's $m$ equal to $m$, which satisfies the required inequality; the Proposition's ${B_0}$ equal to $w+B_1$; the Proposition's $B_1$ equal to $B_2+z$; the Proposition's $B_2$ equal to $B_3$; the Proposition's $B_3$ equal to $B_4$; the Proposition's $B_4$ equal to $B_5$; the Proposition's $S$ equal to $S$; the Proposition's $T$ equal to $T$; the Proposition's $D$ equal to $D$. Then there is a Bohr set $B_6$ with frequency set of size at most $m$ and width $\refconstlittle{8}\frac{\alpha^{4k}}{32m}$ such that for any probability measure $\nu$ supported on $B_6\cap (B_5-B_5)$ we have
\begin{equation*}
\|1_D \ast \nu\|_\infty \geq \langle 1_D,\mu_{T} \ast \mu_{S}\rangle \geq 1-\frac{1}{32}\geq 1- \frac{1}{16}.
\end{equation*}
It follows that
\begin{equation*}
\|1_{A\cap {B_0}} \ast 1_{-(A\cap {B_0})} \ast \nu\|_\infty \geq \left(1 - \frac{1}{16}\right)\left(1+\frac{1}{8}\right)\alpha^2\mu_G({B_0}) \geq \left(1+\frac{1}{32}\right)\alpha^2\mu_G({B_0}), 
\end{equation*}
and we get the result by averaging.
\end{proof}

\begin{proposition}\label{prop.d}
Suppose that $\epsilon,\sigma,\tau  \in (0,1]$,  we have natural numbers $l \geq \newconstbig{L}\log  2\sigma^{-1}\epsilon^{-1}$ and $m \geq \newconstbig{pd}\epsilon^{-2}l^2(\log 2\tau^{-1})(\log 2 \sigma^{-1})$,
\[
\begin{array}{l l}
\mathord{\bullet}\,\mu_G({B_0}-B_1+l(B_2-B_2)) \leq 2\mu_G({B_0}) & \bullet \mu_G({B_0}+B_2)\leq 2\mu_G({B_0})\\[5pt]
\mathord{\bullet}\,\mu_G(B_1+l(B_2-B_2)) \leq 2\mu_G(B_1) & \bullet \mu_G(B_2+mB_3) \leq 2\mu_G(B_2)\\[5pt]
\mathord{\bullet}\,\mu_G(B_4+B_3)\leq (1+\newconstlittle{962}\epsilon\sigma)\mu_G(B_3) &
\end{array}
\]
 and $-S \subset {B_0}$ and $T \subset B_1$ have $\mu_{{B_0}}(-S)\geq \sigma$ and $\mu_{B_1}(T) \geq \tau$, and $D \subset G$.  Then there is a Bohr set $B_5$ with frequency set of size at most $m$ and width $\newconstlittle{8}\frac{\epsilon\sigma}{m}$ such that for any probability measure $\mu$ supported on $B_5\cap (B_4-B_4)$ we have 
\begin{equation*}
\|1_D \ast \mu\|_\infty \geq \langle 1_D,\mu_T\ast \mu_S\rangle -\epsilon.
\end{equation*}
\end{proposition}
\begin{proof}
Let $\refconstbig{pd}>1$ be absolute such that
\begin{equation}\label{eqn.hi}
\refconstbig{pd}\epsilon^{-2}l^2(\log 2\tau^{-1})(\log^3 2 \sigma^{-1})\geq  4\refconstbig{llc}\log 2(\sigma/2)^{-32\refconstbig{csl}l^2\epsilon^{-2}\log_2 2\tau^{-1}}
\end{equation}
for all $\epsilon,\sigma,\tau \in (0,1]$. Let $\refconstbig{L}>1$ be such that $l \geq \refconstbig{L}\log 2\sigma^{-1}\epsilon^{-1}$ implies $2^{l-1} \geq \frac{4}{\sqrt{2}-1}\sigma^{-1/2}\epsilon^{-1}$ for all $\epsilon,\sigma \in (0,1]$.  Let $\refconstlittle{962}=\frac{1}{8\refconstbig{12}}$ and $\refconstlittle{8}=\frac{1}{8}$.

Apply Lemma \ref{lem.crosis} -- a local version of the Croot-Sisask almost periodicity lemma -- with the Lemma's $T$ equal to $B_1+(l-1)(B_2-B_2)$; the Lemma's ${B_0}$ equal to $B_2$; the Lemma's $L$ equal to $2$; the Lemma's $S$ equal to $-S$; the Lemma's $K$ equal to $2\sigma^{-1}$; the Lemma's $p$ equal to $2\log_2 2\tau^{-1}$; and the Lemma's $\epsilon$ equal to $\frac{\epsilon}{4l}$.

This gives us a $t \in B_2$ and $X \subset B_2-t$ with $\mu_{B_2-t}(X) \geq (\sigma/2)^{64\refconstbig{csl}l^2\epsilon^{-2}\log_2 2\tau^{-1}}$ such that
\begin{equation*}
\|\rho_{x}(1_D \ast \mu_{-S} ) - 1_D \ast \mu_{-S} \|_{L_p(\mu_{B_1+(l-1)(B_2-B_2)})} \leq \frac{\epsilon}{4l} \text{ for all } x \in X,
\end{equation*}
where $\rho_x(f)(y)=f(y+x)$. For $x_1,\dots,x_l \in X$ we have $x_1+\cdots + x_j \in l(B_2-B_2)$ for all $j \leq l$ and so by the triangle inequality we have
\begin{align*}
& \|\rho_{x_1+\cdots + x_l}(1_D \ast \mu_{-S} ) - 1_D \ast \mu_{-S} \|_{L_p(\mu_{B_1})}\\ &\qquad\leq \sum_{j=0}^{l-1}{ \|\rho_{x_1+\cdots + x_{j+1}}(1_D \ast \mu_{-S} ) - \rho_{x_1+\cdots + x_{j}}(1_D \ast \mu_{-S}) \|_{L_p(\mu_{B_1})}}\\ &\qquad= \sum_{j=0}^{l-1}{ \|\rho_{x_{j+1}}(1_D \ast \mu_{-S} ) - 1_D \ast \mu_{-S} \|_{L_p(\mu_{B_1+x_1+\cdots + x_j})}}\\&\qquad\leq \sum_{j=0}^{l-1}{\left(\frac{\mu_G(B_1+l(B_2-B_2))}{\mu_G(B_1)}\right)^{1/p} \|\rho_{x_{j+1}}(1_D \ast \mu_{-S} ) - 1_D \ast \mu_{-S} \|_{L_p(\mu_{B_1+(l-1)(B_2-B_2)})}}.
\end{align*}
Integrating this we have
\begin{equation*}
\|1_D \ast \mu_{-S} \ast \overbrace{\mu_{-X} \ast \cdots \ast \mu_{-X}}^{l \text{ times}} - 1_D \ast \mu_{-S} \|_{L_p(\mu_{B_1})} \leq \frac{\epsilon}{2\sqrt{2}}.
\end{equation*}
Since $\|1_T\|_{L_{p/(p-1)}(\mu_{B_1})} \leq \mu_{B_1}(T)^{1-1/p} \leq \mu_{B_1}(T)\tau^{-1/2\log_22\tau^{-1}}\leq \sqrt{2}\mu_{B_1}(T)$ we have
\begin{equation*}
|\langle 1_{D}\ast \mu_{-S}\ast \overbrace{\mu_{-X} \ast \cdots \ast \mu_{-X}}^{l \text{ times}}  , 1_T\rangle_{L_2(\mu_{B_1})} - \langle 1_D\ast \mu_{-S},1_T\rangle_{L_2(\mu_{B_1})}| \leq \frac{1}{2}\epsilon \mu_{B_1}(T).
\end{equation*}
This rearranges to
\begin{equation}\label{eqn.ab}
|\langle 1_{D}, 1_T \ast \overbrace{\mu_X\ast \cdots \ast \mu_X}^{l \text{ times}}\ast \mu_S\rangle_{L_2(\mu_G)} - \langle 1_D,  1_T\ast \mu_S\rangle_{L_2(\mu_G)} | \leq \frac{1}{2}\epsilon \mu_{G}(T).
\end{equation}
Apply Lemma \ref{lem.localchang} -- a local version of Chang's theorem -- with the Lemma's $\epsilon$ equal to $\frac{1}{2}$; the Lemma's $\alpha$ equal to $\mu_{B_2-t}(X)$; the Lemma's $\eta$ equal to $\frac{1}{8\refconstbig{12}}\epsilon\sigma$; the Lemma's $\delta$ equal to $\frac{\epsilon\sigma}{8m}$; the Lemma's $k$ equal to $m$, which satisfies the hypothesis of the Lemma by (\ref{eqn.hi}); the Lemma's ${B_0}$ equal to $B_2$; the Lemma's $B_1$ equal to $B_3$; the Lemma's $B_2$ equal to $B_4$; and the Lemma's $A$ equal to $X+t$. This gives a Bohr set $B_5$ with frequency set of size at most $m$ and width $\frac{\epsilon \sigma}{8m}$ such that
\begin{equation*}
|1-\gamma(x)| \leq \refconstbig{12}\cdot \frac{1}{8\refconstbig{12}}\epsilon \sigma + m \cdot \frac{\epsilon\sigma }{8m}=\frac{1}{4}\epsilon\sigma \text{ for all } x \in (B_4-B_4)\cap B_5
\end{equation*}
whenever $|(1_{X+t}\dd\mu_{B_2})^\wedge(\gamma)| \geq \frac{1}{2}\mu_{B_2}(X+t)$. Since $X+t \subset B_2$, this last condition is equivalent to $|\wh{\mu_X}(\gamma)| \geq \frac{1}{2}$.

Write $\Delta:=\{\gamma:|\wh{\mu_X}(\gamma)| \geq \frac{1}{2}\}$. Suppose that $\mu$ is a probability measure supported on $(B_4-B_4)\cap B_5$. Then $|\wh{\mu}(\gamma)-1| \leq \frac{1}{4}\epsilon \sigma$ for all $\gamma \in \Delta$ by the triangle inequality. By Plancherel's Theorem
\begin{align*}
&\left|\langle 1_{D}\ast \mu , 1_T \ast \overbrace{\mu_X\ast \cdots \ast \mu_X}^{l \text{ times}}\ast \mu_S\rangle_{L_2(\mu_{G})} - \langle 1_{D}, 1_T \ast \overbrace{\mu_X\ast \cdots \ast \mu_X}^{l \text{ times}} \ast \mu_S\rangle_{L_2(\mu_{G})} \right|\\ &\qquad \leq \sum_{\gamma\in \Delta}{|\wh{\mu}(\gamma)-1||\wh{1_D}(\gamma)\wh{1_T}(\gamma)\wh{\mu_X}(\gamma)^l \wh{\mu_S}(\gamma)|}
+\sum_{\gamma \not \in \Delta }{|\wh{\mu}(\gamma)-1||\wh{1_D}(\gamma)\wh{1_T}(\gamma)\wh{\mu_X}(\gamma)^l \wh{\mu_S}(\gamma)|}\\ &\qquad \leq \frac{1}{4}\epsilon \sigma\sum_{\gamma\in \Delta}{|\wh{1_D}(\gamma)\wh{1_T}(\gamma)\wh{\mu_X}(\gamma)^l \wh{\mu_S}(\gamma)|}
+2\sum_{\gamma \not \in \Delta }{|\wh{1_D}(\gamma)\wh{1_T}(\gamma)\wh{\mu_X}(\gamma)^l \wh{\mu_S}(\gamma)|}\\ &\qquad \leq  \frac{1}{4}\epsilon \sigma\sum_{\gamma\in \wh{G}}{|\wh{1_D}(\gamma)\wh{1_T}(\gamma) \wh{\mu_S}(\gamma)|}+\frac{1}{2^{l-1}}\sum_{\gamma \in \wh{G} }{|\wh{1_D}(\gamma)\wh{1_T}(\gamma)\wh{\mu_S}(\gamma)|}\\ &\qquad\leq \frac{1}{4}\epsilon \sigma \mu_G(T)\sqrt{\mu_G(D)\mu_G(S)^{-1}} +  \frac{\mu_G(T)\sqrt{\mu_G(D)\mu_G(S)^{-1}}}{2^{l-1}}.
\end{align*}
We may certainly assume that $D \subset S+T \subset -{B_0}+B_1+(l-1)(B_2-B_2)\subset -B_0+B_1+l(B_2-B_2)$, so $\mu_G(D) \leq 2\mu_G({B_0})$. Since $2^{1-l} \leq \frac{\sqrt{2}-1}{4} \epsilon \sqrt{\sigma}$ and $\sigma \mu_G(B_0)=\mu_G(S)$ we can therefore say the right hand side above is at most $\frac{1}{2}\epsilon \mu_G(T)$. Combining this with (\ref{eqn.ab}) we have
\begin{equation*}
\left|\langle 1_{D}\ast \mu, 1_T \ast \overbrace{\mu_X\ast \cdots \ast \mu_X}^{l \text{ times}}\ast \mu_S\rangle_{L_2(\mu_G)} - \langle 1_D, 1_T \ast \mu_S\rangle_{L_2(\mu_G)}\right| \leq \epsilon \mu_G(T).
\end{equation*}
The result follows.
\end{proof}

\begin{proof}[Proof of Theorem \ref{thm.keycount}]
 Let $p=O_{a,b}(N)$ be a prime larger than $(2|a|+|b|)N$ (possible by Bertrand's postulate \cite[Theorem 417, p343]{Hardy:1960aa}); let $G:=\Z/p\Z$; let $\Gamma_0=\{G \rightarrow S^1; x \mapsto \exp(2\pi i x/p)\}$, and $\delta_0=\Omega(1/(2|a|+|b|))$ be such that $B(\Gamma_0,\delta_0)=\{z+p\Z: z \in  \{-N,\dots,N\}\}$.  The quotient map $\Z \rightarrow \Z/p\Z$ maps the $r$-colouring $\mathcal{C}$ to an $r$-colouring of $\{z + p\Z: -N \leq z\leq N\}$ which restricts to an (at most) $r$-colouring $\mathcal{C}'$ of $B(\Gamma_0,\delta_0)$. Furthermore, if $x,y,z \in B(\Gamma_0,\delta_0)$ are all the same colour in $\mathcal{C}'$ and have $ax-ay\equiv bz \pmod p$, then there are $x',y',z' \in \{-N,\dots,N\}$ congruent to $x$, $y$, and $z$ respectively, all of the same colour in $\mathcal{C}$ and with $ax-ay=bz$. Hence there are at least as many monochromatic solutions to $ax-ay=bz$ in $\mathcal{C}$ as there are to $ax-ay\equiv bz \pmod p$ in $\mathcal{C}'$. We shall prove the result by giving a lower bound on the latter.
 
We proceed iteratively to construct frequency sets $\Gamma_0,\Gamma_1,\dots$ of sizes $d_0 \leq d_1\leq \dots$ respectively, and widths $1 \geq \delta_0 \geq \delta_1 \geq \dots >0$, and so that the corresponding Bohr sets are nested, meaning $B(\Gamma_0,\delta_0) \supset B(\Gamma_1,\delta_1)\supset \dots$. Note that we do \emph{not} require the $\Gamma_i$s themselves to be nested.

Write $\mathcal{C}'=\{A_1,\dots,A_r\}$ and $\cdot$ for the scalar multiplication of $\Z/p\Z$ on itself. For each step $i$ of the iteration, and $j \in [r]$ write
\begin{equation*}
S_{i,j}:=\min{\{\|1_{a \cdot A_j} \ast \mu\|_{\infty} : \mu \text{ is a probability measure supported in } B(\Gamma_i,\delta_i)\}}.
\end{equation*}
Since $B(\Gamma_i,\delta_i) \subset B(\Gamma_k,\delta_k)$ whenever $i \geq k$ we have
\begin{equation}\label{mon2}
S_{i,j} \geq S_{k,j}\text{ whenever }i \geq k, \text{ and } S_{i,j} \leq 1 \text{ for all }i.
\end{equation}

At step $i$ of the iteration we shall show that either
\begin{enumerate}
\item\label{cd1} there is some $j \in [r]$ such that
\begin{equation*}
\langle 1_{a\cdot A_j} \ast 1_{-a\cdot A_j},1_{b \cdot A_j}\rangle_{L_2(\mu_G)} \geq \left(\frac{\delta_i}{2rd_i|a||b|}\right)^{O(d_i)};\text{ or }
\end{equation*}
\item \label{cd0} there is $j \in [r]$ such that $S_{i,j} < 1/2r$ and $S_{i+1,j} \geq 1/2r$;
\item\label{cd2} there is $j \in [r]$ such that
\begin{enumerate}
\item\label{cd21} $d_{i+1} \leq d_i + O(\log^82r)$ and $\delta_{i+1} \geq  \exp(-O_{a,b}(\log^22r+\log 2d_i)) \delta_i$ ;
\item\label{cd22} $S_{i+1,j} \geq (1+\Omega(1))S_{i,j}$; and
\item\label{cd23} $S_{i,j} \geq 1/2r$.
\end{enumerate}
\end{enumerate}
We stop the iteration the first time we are in case (\ref{cd1}). By monotonicity (\ref{mon2}) we can be in case (\ref{cd0}) at most once for each $j \in [r]$. Suppose that we are in case (\ref{cd2}) for a particular $j$ at steps $i_1 <i_2<\dots < i_k$ of the iteration. Then by monotonicity (\ref{mon2}) and (\ref{cd22}) \& (\ref{cd23}) we have
\begin{align*}
1 \geq S_{i_k+1,j} \geq (1+\Omega(1))S_{i_k,j} & \geq (1+\Omega(1))S_{i_{k-1}+1,j}\\ & \geq \qquad \dots \qquad  \geq (1+\Omega(1))^kS_{i_1,j} \geq (1+\Omega(1))^k\cdot (1/2r).
\end{align*}
It follows that $k=O(\log 2r)$. Since there are $r$ possible values for $j$ we conclude that we must have terminated the iteration at step $i_0\leq r\cdot (1+O(\log 2r))$, in which case
\begin{equation*}
d_{i_0} =i_0\cdot O(\log^82r) = O(r\log^92r)
\end{equation*}
and
\begin{equation*}
\delta_{i_0} = \exp(-i_0O_{a,b}(\log^22r))=\exp(-O_{a,b}(r\log^32r))
\end{equation*}
by (\ref{cd21}) and the fact that $\delta_0=\Omega_{a,b}(1)$.

Since $|G|^2\langle 1_{a\cdot A_j} \ast 1_{-a\cdot A_j},1_{b \cdot A_j}\rangle_{L_2(\mu_G)} $ is exactly the number of solutions to $ax-ay\equiv bz \pmod p$ for $x,y,z \in A_j$ we have the conclusion of Theorem \ref{thm.keycount} from (\ref{cd1}).

It remains to show that at each stage of the iteration we are either in case (\ref{cd1}), (\ref{cd0}) or (\ref{cd2}). Suppose we are at step $i$ of the iteration.

Anticipating an application of Lemma \ref{lem.itlem} and the many hypotheses we need to ensure there we apply Lemma \ref{lem.regular} repeatedly. Let $k :=\lceil \refconstbig{spec}\log 4r\rceil$, $l :=\lceil  \refconstbig{spec3}k \log 4r\rceil$, and  $m := \lceil \refconstbig{spec2}l^2k^2\log^24r\rceil$.

By Lemma \ref{lem.regular} there is a Bohr set $B_0 \subset B(\Gamma_i,\delta_i)$ of width $\Omega(\delta_i)$ and frequency set $\Gamma_i$, and another $B_1'$ of width $\Omega(\delta_i/r^2|a||b|d_i)$ and frequency set $\Gamma_i$ such that
\begin{equation*}
\mu_G(B_0+4abB_1')\leq \left(1+\min\left\{\frac{\refconstlittle{me}}{2r},\frac{\refconstlittle{16}}{4r^2}\right\}\right)\mu_G(B_0) \text{ and }4abB_1' \subset B_0.
\end{equation*}
Let $\Gamma_i':=\{G \rightarrow S^1; x \mapsto \gamma(b^{-1}x): \gamma \in \Gamma_i\}$ and $B_1''$ be the Bohr set with frequency set $\Gamma_i'$ and the same width as $B_1'$. Then $B_1''=b \cdot B_1'$ -- this equality is key and we will use it again -- and so $4B_1'' \subset B_0$ by the triangle inequality. 

By Lemma \ref{lem.regular} again there is a Bohr set $B_1\subset B_1''$ of width $\Omega(\delta_i/r^2|a||b|d_i)$ and frequency set $\Gamma_i'$, and another $B_2'$ of width $\Omega(\delta_i/r^3l|a||b|d_i^2)$ and frequency set $\Gamma_i'$ such that
\begin{equation}\label{eqn.b2size}
\mu_G(B_1+2(l+1)B_2')\leq \left(1+\frac{1}{4r}\right)\mu_G(B_1) \text{ and }2(l+1)B_2' \subset B_1.
\end{equation}
By Lemma \ref{lem.regular} again there is a Bohr set $B_2 \subset B_2'$ of width $\Omega(\delta_i/r^3l|a||b|d_i^2)$ and frequency set $\Gamma_i'$, and another $B_3'$ of width $\Omega(\delta_i/r^3l^2|a||b|d_i^3)$ and frequency set $\Gamma_i'$  such that
\begin{equation*}
\mu_G(B_2+2lB_3')\leq 2\mu_G(B_2) \text{ and }2lB_3' \subset B_2.
\end{equation*}
By Lemma \ref{lem.regular} again there is a Bohr set $B_3 \subset B_3'$ of width $\Omega(\delta_i/r^3l^2|a||b|d_i^3)$ and frequency set $\Gamma_i'$, and another $B_4'$ of width $\Omega(\delta_i/r^3l^2m|a||b|d_i^4)$ and frequency set $\Gamma_i'$  such that
\begin{equation*}
\mu_G(B_3+mB_4')\leq 2\mu_G(B_3) \text{ and } mB_4' \subset B_3.
\end{equation*}
Finally, by Lemma \ref{lem.regular} again there is a Bohr set $B_4 \subset B_4'$ of width $\Omega(\delta_i/r^3l^2m|a||b|d_i^4)$ and frequency set $\Gamma_i'$, and another $B_5'$ of width $\Omega(\delta_i/(2r)^{3+4k}l^2m|a||b|d_i^5)$ and frequency set $\Gamma_i'$  such that
\begin{equation*}
\mu_G(B_4+B_5')\leq \left(1+\refconstlittle{96}(2r)^{-4k}\right)\mu_G(B_4) \text{ and } B_5' \subset B_4.
\end{equation*}

The measure $\mu:=\wt{\mu_{B_1}} \ast \wt{\mu_{B_2}} \ast \mu_{B_2} \ast \mu_{B_1}$ is supported in $4B_1 \subset 4(b\cdot B_1') \subset b\cdot B_0$. Since $\mathcal{C}'$ is a colouring of $B_0$, $\{b\cdot A_1,\dots,b\cdot A_r\}$ is a colouring of $b\cdot B_0$ and hence of $4 B_1$. Thus, by averaging there is $j$ such that $\mu(b \cdot A_j) \geq 1/r$.

Suppose that $S_{i,j} < 1/2r$. Since $\mu(b \cdot A_j) \geq 1/r$ we have $\|1_{b \cdot A_j} \ast \mu_{B_1}\|_{\infty} \geq 1/r$, and hence $\|1_{a\cdot A_j} \ast \mu_{(ab^{-1})\cdot B_1}\|_\infty \geq 1/r$. We apply Lemma \ref{lem.large} with the Lemma's $\eta$ set to $1/4r$, the Lemma's $B_0$ equal to $(ab^{-1})\cdot B_1$ and the Lemma's $B_1$ equal to $(ab^{-1})\cdot B_2$ which satisfies the required inequality by (\ref{eqn.b2size}) since the map $x \mapsto ab^{-1}x$ is a bijection of $G$. Set $\Gamma_{i+1}:=\{G \rightarrow S^1; x \mapsto \gamma(a^{-1}x): \gamma \in \Gamma_i\}$ and let $\delta_{i+1}$ be the width of $B_2$. Then $B(\Gamma_{i+1},\delta_{i+1})=ab^{-1}\cdot B_2$ and so we have (\ref{cd0}), and moreover $ab^{-1}\cdot B_2 \subset ab^{-1}\cdot B_1 \subset a\cdot B_1' \subset B_0$ since $B_1 \subset B_1'' = b\cdot B_1'$. It follows that $B(\Gamma_{i+1},\delta_{i+1}) \subset B(\Gamma_i,\delta_i)$ as required.

On the other hand, suppose that $S_{i,j} \geq 1/2r$. Let $x_i \in G$ be such that $\|1_{a \cdot A_j} \ast \mu_{B_0}\|_\infty = \mu_{B_0}(x_i-a\cdot A_j)$. Apply Lemma \ref{lem.itlem} with the Lemma's $\alpha:=\mu_{B_0}(x_i-a\cdot A_j)\geq 1/2r$, $\delta = \mu(b\cdot A_j)\geq 1/r$, $k$, $m$, $l$ and $B_0$,\dots,$B_5$ as defined. If the first conclusion of Lemma \ref{lem.itlem} holds then
\begin{align*}
\frac{1}{8r^3}\mu_G(B_0) & \leq \langle 1_{x_i-a\cdot A_j} \ast 1_{-(x_i - a\cdot A_j)},1_{b \cdot A_j}\rangle_{L_2(\mu)}\\ & =  \langle 1_{a\cdot A_j} \ast 1_{- a\cdot A_j},1_{b \cdot A_j}\rangle_{L_2(\mu)} \leq \mu_G(B_1)^{-2}\mu_G(B_2)^{-2}\langle  1_{a\cdot A_j} \ast 1_{- a\cdot A_j},1_{b \cdot A_j}\rangle_{L_2(\mu_G)},
\end{align*}
and we are in case (\ref{cd1}) by the size bound for Bohr sets (Lemma \ref{lem.bohrsize}). If the second conclusion of Lemma \ref{lem.itlem} holds then let $\Gamma_{i+1}$ equal $\Gamma_i$ union the frequency set of the Bohr set $B_6$ and let $\delta_{i+1}$ be the minimum of the width of $B_5$ and $\refconstlittle{128}(1/2r)^{4k}/m$. It follows that
\begin{equation*}
d_{i+1} \leq d_i + m = d_i+O(\log^82r)\text{ and }\delta_{i+1} = \exp(-O_{a,b}(\log^22r + \log 2d_i))\delta_i,
\end{equation*}
 and since $B(\Gamma_{i+1},\delta_{i+1}) \subset B_6 \cap (B_5-B_5)$, 
 \begin{equation*}
 S_{i+1,j}\geq (1+\refconstlittle{32})\mu_{B_0}(x_i-a \cdot A_j)=(1+\refconstlittle{32})\|1_{a \cdot A_j} \ast \mu_{B_0}\|_\infty \geq (1+\refconstlittle{32})S_{i,j}.
 \end{equation*}
It follows that (\ref{cd21}) and (\ref{cd22}) hold (and by supposition (\ref{cd23})), completing case (\ref{cd2}) as required. The result is proved.
\end{proof}

\appendix

\section{Bohr sets}\label{ap.bohr}

Bohr sets have well-controlled sizes; the following is a version of \cite[Lemma 4.19, p188]{taovu::} (see also \cite[\S2]{bou::5}) proved with a standard covering argument.
\begin{lemma}\label{lem.bohrsize}
Suppose that $\Gamma$ is a set of $d$ characters. Then
\begin{equation*}
\mu_G(B(\Gamma,2\delta)) \leq \exp(\newconstbig{grow}d)\mu_G(B(\Gamma,\delta))\text{ and }\mu_G(B(\Gamma,\delta))\geq (\delta/2)^{O(d)}.
\end{equation*}
\end{lemma}
\begin{proof}
Let $\Gamma=\{\gamma_1,\dots,\gamma_d\}$ and define the map $\phi:G \rightarrow \C^d; x \mapsto (\gamma_i(x))_{i=1}^d$. Write $D_\delta:=\{z \in \C: |z-1| \leq \delta\}$ and $Q_\delta:=D_\delta^d$ so that $\mu(Q_\delta)=(\pi \delta^2)^d$ where $\mu$ is Lebesgue measure on $\C^d$. Let $T \subset Q_{2\delta}$ be maximal such that for $s,t \in T$ distinct, $(t+Q_{\delta/4})\cap (s+Q_{\delta/4})= \emptyset$ --  in words $T \subset Q_{2\delta}$ is maximal $Q_{\delta/4}$-separated.

Then $\mu(Q_{\delta/4})|T| \leq \mu(Q_{2\delta+\delta/2})$ and $Q_{2\delta} \subset T+Q_{\delta/2}$ by the triangle inequality. It follows that $|T| \leq 100^{d}$, and
\begin{equation*}
\mu_G(B(\Gamma,2\delta)) = \mu_G(\phi^{-1}(Q_{2\delta})) \leq \mu_G(\phi^{-1}(T+Q_{\delta/2})) \leq 100^{d}\sup_{t \in T}{\mu_G(\phi^{-1}(t+Q_{\delta/2}))}.
\end{equation*}
 If $x,y \in \phi^{-1}(t+Q_{\delta/2})$ then $\phi(x-y) \in Q_\delta$ and so $\mu_G(\phi^{-1}(t+Q_{\delta/2})) \leq \mu_G(B(\Gamma,\delta))$. The result is proved, and the size bound follows by iterated application of the growth bound since $\mu_G(B(\Gamma,2))=1$.
\end{proof}
This growth lets us identify pairs of Bohr sets that produce approximate actions as discussed in \S\ref{sec.mil}:
\begin{lemma}\label{lem.regular} Suppose $\delta \in (0,1]$ and $\Gamma$ is a set of $d$ characters, and $l \in \N$, $\eta \in (0,1]$. Then there is $\delta_* \in [\delta/2,\delta]$ and $\delta' \in  \Omega(\delta \eta/ld)$ such that
\begin{equation*}
\mu_G(B(\Gamma,\delta_*)+l B(\Gamma,\delta'))\leq (1+\eta)\mu_G(B(\Gamma,\delta_*)) \text{ and } lB(\Gamma,\delta') \subset B(\Gamma,\delta_*).
\end{equation*}
\end{lemma}
\begin{proof}
Let $k=O(\eta^{-1})$ be a natural number such that $\exp(\refconstbig{grow}/k) \leq 1+\eta$, and let $\delta'=\delta/2lkd$ and $\delta_i=\delta/2+il\delta'$ for $0 \leq i \leq kd-1$. Then 
\begin{equation*}
\prod_{i=0}^{kd-1}{\frac{\mu_G(B(\Gamma,\delta_i+l\delta'))}{\mu_G(B(\Gamma,\delta_i))}} \leq \frac{\mu_G(B(\Gamma,\delta))}{\mu_G(B(\Gamma,\delta/2))} \leq \exp(\refconstbig{grow}d).
\end{equation*}
By averaging there is some $\delta_*=\delta_i$ such that 
\begin{equation*}
\frac{\mu_G(B(\Gamma,\delta_*+l\delta'))}{\mu_G(B(\Gamma,\delta_*))} \leq 1+\eta.
\end{equation*}
The result follows by the triangle inequality.
\end{proof}
There is a stronger result in \cite[Lemma 4.24, p192]{taovu::} (see also \cite[\S3]{bou::5}) essentially showing that $\delta_*$ can be chosen to be independent of $\eta$. 

The following lemma is useful for relating large values of the Fourier transform to Bohr sets, as we will need in Appendix \ref{ap.chang}:
\begin{lemma}\label{lem.inv}
Suppose that $\eta,\kappa \in (0,1]$, $\mu_G(B_1+{B_0}) \leq (1+\eta)\mu_G({B_0})$ and $|\wh{\mu_{{B_0}}}(\gamma)| \geq \kappa$. Then $|1-\gamma(x)|\leq 2\eta\kappa^{-1}$ for all $x \in B_1-B_1$.
\end{lemma}
\begin{proof} For $x \in B_1-B_1$ there are $s,t \in B_1$ such that $x=s-t$, so $|\gamma(s)-\gamma(t)|=|1-\gamma(x)|$. Now
\begin{equation*}
|\gamma(s)-\gamma(t)||\wh{\mu_{B_0}}(\gamma)| = |\wh{\mu_{t+{B_0}}}(\gamma) - \wh{\mu_{s+{B_0}}}(\gamma)|  \leq \frac{\mu_G((t+{B_0}) \triangle (s+{B_0}))}{\mu_G({B_0})} \leq 2\eta.
\end{equation*}
Dividing by $\kappa$ gives the result.
\end{proof}
By a related argument Bohr sets support a technical device for replacing density by a sort of `hereditary density' which is better behaved:
\begin{lemma}\label{lem.large}
Suppose that $\alpha,\eta \in (0,1]$, $\mu_G(B_0+B_1) \leq (1+\eta)\mu_G(B_0)$, and $\|1_A \ast \mu_{B_0}\|_\infty \geq \alpha$. Then $\min{\{\|1_A \ast \mu\|_\infty : \mu \text{ is a probability measure supported in } B_1-B_1\}} \geq \alpha - 2\eta$.
\end{lemma}
\begin{proof}
If $x$ is in the support of $\mu$ then there is $z,w \in B_1$ such that $x=z-w$ and
\begin{equation*}
\|\mu_{B_0+x} - \mu_{B_0}\|=\|\mu_{B_0+z} - \mu_{B_0+w}\| \leq \frac{\mu_{G}( (B_0+z) \triangle (B_0+w))}{\mu_G(B_0)} \leq 2\eta.
\end{equation*}
By the triangle inequality we conclude that $\|\mu_{B_0} \ast \mu - \mu_{B_0}\|\leq 2\eta$ and hence
\begin{equation*}
\|1_A \ast \mu\|_\infty \geq \|1_A \ast \mu_{B_0}\ast \mu\|_\infty \geq \|1_A \ast \mu_{B_0}\|_\infty - 2\eta.
\end{equation*}
 The result is proved.
\end{proof}

\section{A local version of spectral positivity}\label{ap.specpos}

We need the following version of the spectral positivity arguments of Kelley and Meka \cite[\S5.2]{kelmek::0}.
\begin{lemma}\label{lem.specpos}
Suppose that $\eta,\alpha,\epsilon,\delta \in (0,1]$, $k$ is a natural number, $\mu_G(B_1+{B_0}) \leq (1+\eta)\mu_G({B_0})$, $\mu_{B_0}(A)= \alpha$, $\mu$ is a positive-definite probability measure supported on $B_1-B_1$, $\mu(D)=\delta$, and
\begin{equation*}
|\langle 1_{A \cap {B_0}} \ast 1_{-(A\cap {B_0})},1_D\rangle_{L_2(\mu)}-\delta \alpha^2\mu_G({B_0})| \geq \epsilon\delta \alpha^2\mu_G({B_0}).
\end{equation*}
Then
\begin{equation*}
\|1_{A\cap {B_0}} \ast 1_{-(A\cap {B_0})}\|_{L_{2k}(\mu)} \geq ((1+\epsilon)(\delta/2)^{1/2k} - \newconstbig{specpos}\eta \alpha^{-1})\alpha^{2}\mu_G({B_0}).
\end{equation*}
\end{lemma}
\begin{proof}
It is enough to show this when $A$ is a subset of ${B_0}$, so we suppose this is the case. Write $f:=(1_A - \alpha1_{B_0}) \ast (1_{-A}-\alpha1_{-{B_0}})$, so that we have the pointwise estimate
\begin{equation}\label{eqn.pointwise}
f(x)=1_A \ast 1_{-A}(x) -\alpha^2\mu_G({B_0})+O(\eta\alpha \mu_G({B_0}))\text{ for all }x \in B_1-B_1.
\end{equation}
Hence
\begin{equation*}
|\langle f,1_D\rangle_{L_2(\mu)}|\geq \epsilon \delta \alpha^2\mu_G({B_0}) - O(\eta \alpha \delta \mu_G({B_0})) \geq \varepsilon \delta \alpha^2\mu_G({B_0})
\end{equation*}
where $\varepsilon = \epsilon -O(\eta \alpha^{-1})$. For $j\in \N^*$ apply H{\"o}lder's inequality -- this is \cite[Step 1, pp23]{Bloom:2023aa} -- to get that
\begin{equation*}
\|f\|_{L_{2j}(\mu)}^{2j} \delta^{2j-1}=\|f\|_{L_{2j}(\mu)}^{2j}\|1_D\|_{L_{\frac{2j}{2j-1}}(\mu)}^{2j} \geq |\langle f,1_D\rangle_{L_2(\mu)}|^{2j} \geq \varepsilon^{2j}\delta^{2j} \alpha^{4j}\mu_G({B_0})^{2j},
\end{equation*}
which rearranges to $\int{f^{2j}\dd\mu}  \geq \varepsilon^{2j}\delta \alpha^{4j}\mu_G({B_0})^{2j}$. On the other hand, by design $\wh{f} \geq 0$, and by hypothesis $\wh{\mu}\geq 0$, so $\int{f^{2j+1}\dd\mu} \geq 0$. By the Binomial formula these then give
\begin{align*}
\int{(f+\alpha^2\mu_G({B_0}))^{2k} \dd\mu} &= \sum_{i=0}^{2k}{\binom{2k}{i}\int{f^i\dd\mu}\alpha^{4k-2i}\mu_G({B_0})^{2k-i}}\\ & \geq \sum_{j=0}^k{\binom{2k}{2j}\varepsilon^{2j}\delta \alpha^{4j}\mu_G({B_0})^{2j}\alpha^{4k-4j}\mu_G({B_0})^{2k-2j}}\\ &=\frac{(1+\varepsilon)^{2k}+(1-\varepsilon)^{2k}}{2}\delta \alpha^{4k}\mu_G({B_0})^{2k}\\ & \qquad \qquad \geq (1+\varepsilon)^{2k}\cdot \frac{\delta}{2}\cdot \alpha^{4k}\mu_G({B_0})^{2k}.
\end{align*}
Now from (\ref{eqn.pointwise}) we have $\|f + \alpha^2\mu_G({B_0})\|_{L_{2k}(\mu)} \leq \|1_A \ast 1_{-A}\|_{L_{2k}(\mu)} + O(\eta \alpha \mu_G({B_0}))$ and the result follows.
\end{proof}

\section{A local version of sifting}\label{ap.sift}

We need the following version of the sifting lemma of Kelley and Meka \cite[\S 4.1]{kelmek::0}. This is proved in the same way as the `combinatorial lemma' of Gowers \cite[Lemma 11]{gow::4} with `five' there replaced by $2k$ and also with the addition of the $\epsilon$ parameter. The reason this result is so useful to Kelley and Meka, and hence to us, is exactly because of this $\epsilon$ parameter.
\begin{lemma}\label{lem.drc}
Suppose that $\alpha,\epsilon,\kappa\in (0,1]$, $k \geq \newconstbig{rdc}\epsilon^{-1}\log 2\kappa^{-1}$ is a natural,  $\mu_G(B_2+B_1+{B_0}) \leq \left(1+\frac{\epsilon \alpha^2}{4}\right)\mu_G({B_0})$, $\mu_{B_0}(A)=\alpha$, and
\begin{equation*}
\|1_{A\cap {B_0}} \ast 1_{-(A\cap {B_0})}\|_{L_{2k}( \wt{\mu_{B_1}}\ast \wt{\mu_{B_2}}\ast \mu_{B_2} \ast \mu_{B_1} )} \geq (1+\epsilon)\alpha^{2}\mu_G({B_0}).
\end{equation*}
Then there are sets $T$ and $S$, and elements $z$ and $w$ with
\begin{equation*}
\mu_{B_2+z}(T) ,\mu_{-B_1-w}(S) \geq \alpha^{4k} \text{ and }\langle 1_D,\mu_{T} \ast \mu_{S}\rangle \geq 1-\kappa
\end{equation*}
where
\begin{equation*}
D=\left\{u \in -B_1-B_2+B_2+B_1: 1_{A\cap {B_0}} \ast 1_{-(A\cap {B_0})}(u)>\left(1+\frac{\epsilon}{2}\right)\alpha^2\mu_G({B_0})\right\}.
\end{equation*}
\end{lemma}
\begin{proof}
Let $\refconstbig{rdc}>1$ be absolute such that $(1+\frac{1}{2}\epsilon)^{2k}(1+\frac{3}{4}\epsilon)^{-2k} \leq \frac{1}{2}\kappa$ for all $\epsilon,\kappa \in (0,1]$ whenever $k \geq \refconstbig{rdc}\epsilon^{-1}\log 2\kappa^{-1}$.

It is enough to show this when $A \subset {B_0}$ and so we assume this is the case. We may also translate $B_2+B_1$ such that it contains $0_G$.  Let $x_1,\dots,x_{2k}$ be chosen independently and uniformly at random from ${B_0}$. Now $\P(x,y \in x_i-A)=\mu_{B_0}((x+A) \cap (y+A))$ and so for $x,y \in B_2+B_1$,
\begin{align*}
 \frac{1_A \ast 1_{-A}(-x+y)}{\mu_G({B_0})} \geq \P(x,y \in x_i-A) & \geq  \frac{1_A \ast 1_{-A}(-x+y)}{\mu_G({B_0})} -\frac{\mu_G(B_2+B_1+{B_0})-\mu_G({B_0})}{\mu_G({B_0})} \\ & \geq \frac{1_A \ast 1_{-A}(-x+y)}{\mu_G({B_0})}-\frac{1}{4}\epsilon \alpha^2.
\end{align*}
Setting $A':=(x_1-A) \cap \cdots \cap (x_{2k}-A)$, by linearity of expectation we have
\begin{align}
\nonumber &\E{\int{ \mu_{B_2+z}(A')\mu_{w+B_1}(A') \dd \mu_{B_1}(z)\dd\mu_{B_2}(w)}}\\ \nonumber & \qquad \qquad =\E{\mu_{B_2}\ast \mu_{B_1}(A')\mu_{B_2} \ast \mu_{B_1}(A')}\\ \nonumber & \qquad \qquad = \int{\prod_{i=1}^{2k}{\P(x,y \in A-x_i)}\dd\mu_{B_2} \ast \mu_{B_1}(x)\dd\mu_{B_2} \ast \mu_{B_1}(y)}\\  \label{eqn.sizelower} & \qquad \qquad  \geq \left(\frac{\|1_A \ast 1_{-A}\|_{L_{2k}( \wt{\mu_{B_1}}\ast \wt{\mu_{B_2}}\ast \mu_{B_2} \ast \mu_{B_1} )}}{\mu_G({B_0})} -\frac{1}{4}\epsilon \alpha^2\right)^{2k} \geq \left(1+\frac{3}{4}\epsilon\right)^{2k}\alpha^{4k}.
\end{align}
On the other hand, for $x,y \in B_2+B_1$ with $-x+y \not \in D$ we have 
 \begin{equation*}
 \P((x,y) \in (A')^2) \leq \left(1+\frac{1}{2}\epsilon\right)^{2k}\alpha^{4k}.
 \end{equation*}
 Hence
 \begin{align}
\nonumber &  \left(1+\frac{1}{2}\epsilon\right)^{2k}\alpha^{4k}\\ \nonumber &\qquad  \geq \E{ \int{1_{D^c}(-x+y)1_{A'}(x)1_{A'}(y)\dd\mu_{B_2}\ast \mu_{B_1}(x)\dd \mu_{B_2}\ast \mu_{B_1}(y)}}\\ \nonumber  &\qquad   =\E{ \int{\int{1_{D^c}(-x+y)1_{A'}(x)1_{A'}(y)\dd \mu_{B_2+z}(x) \dd \mu_{w+B_1}(y)}\dd \mu_{B_1}(z)\dd \mu_{B_2}(w)}}\\
&\label{eqn.errorlower} \qquad  = \E{ \int{\langle 1_{D^c},\mu_{A' \cap (B_2+z)} \ast \mu_{-(A' \cap (w+B_1))}\rangle\mu_{B_2+z}(A')\mu_{w+B_1}(A') \dd \mu_{B_1}(z)\dd \mu_{B_2}(w)}}.
 \end{align}
Consider the inequality resulting from (\ref{eqn.sizelower}) minus $\kappa^{-1}$ times (\ref{eqn.errorlower}). Since $(1+\frac{1}{2}\epsilon)^{2k}(1+\frac{3}{4}\epsilon)^{-2k} \leq \frac{1}{2}\kappa$, by averaging there is a choice of $x_1,\dots,x_{2k}$, $z$ and $w$ such that
 \begin{equation*}
 \mu_{B_2+z}(A')\mu_{w+B_1}(A')(1 -\kappa^{-1}\langle 1_{D^c},\mu_{A' \cap (B_2+z)} \ast \mu_{-(A' \cap (w+B_1))}\rangle) \geq \frac{1}{2}\left(1+\frac{3}{4}\epsilon\right)^{2k}\alpha^{4k}\geq \alpha^{4k}.
 \end{equation*}
 Set $T:=A' \cap (B_2+z)$ and $S:=-(A' \cap (w+B_1))$, and the result is proved.
\end{proof}

\section{A local version of Chang's theorem}\label{ap.chang}

The result \cite[Lemma 3.1]{cha::0} was popularised by Green \cite{gre::6}. We shall need the following local version.

\begin{lemma}\label{lem.localchang}
Suppose $\epsilon,\alpha,\eta,\delta \in (0,1]$, $k \geq \newconstbig{llc}\epsilon^{-2}\log 2\alpha^{-1}$ is a natural number, $\mu_G({B_0}+kB_1) \leq 2 \mu_G({B_0})$, $\mu_G(B_2+B_1)\leq (1+\eta)\mu_G(B_1)$, and $\mu_{B_0}(A)\geq \alpha$. Then there is a Bohr set $B_3$ with frequency set of size at most $k$ and width $\delta$ such that
\begin{equation}\label{eqn.estimate}
|1-\gamma(x)|\leq \newconstbig{12} \eta+k\delta \textrm{ for all }x \in (B_2-B_2)\cap B_3 
\end{equation}
whenever $|(1_A\dd\mu_{B_0})^\wedge(\gamma)| \geq \epsilon \mu_{B_0}(A)$. 
\end{lemma}
To prove this we make some definitions: given a set of characters $\Lambda$ and a function $\omega:\Lambda \rightarrow \{z \in \C: |z|\leq 1\}$ we define the Riesz product
\begin{equation*}
p_{\omega}:=\prod_{\lambda \in \Lambda}{(1+\Re \omega(\lambda)\lambda)}.
\end{equation*}
For a probability measure $\mu$ the set $\Lambda$ is said to be \emph{$K$-dissociated w.r.t.\ $\mu$} if
\begin{equation*}
\int{p_{\omega}\dd\mu} \leq \exp(K) \textrm{ for all } \omega:\Lambda \rightarrow \{z \in \C: |z|\leq 1\}.
\end{equation*}
We make use of this through the next proposition which is substantially the proof of \cite[Proposition 3.4]{greruz::0} and the duality argument of \cite[\S 5]{gre::6}.
\begin{lemma}\label{lem.changbd}  Suppose that $\alpha,\epsilon \in (0,1]$, $\mu$ a probability measure on $G$,  and $\mu(A)\geq \alpha$. If $\Lambda$ is $1$-dissociated w.r.t.\ $\mu$ and a subset of
\begin{equation*}
\{\gamma \in \wh{G}: |(1_A\dd\mu)^\wedge(\gamma)| \geq \epsilon \mu(A)\},
\end{equation*}
then $|\Lambda| \leq \newconstbig{lcb}\epsilon^{-2}\log 2\alpha^{-1}$.
\end{lemma}
\begin{proof}
For $k \in \N^*$ define the linear operator
\begin{equation*}
T:\ell_2(\Lambda) \rightarrow L_{2k}(\mu); g \mapsto \sum_{\lambda \in \Lambda}{g(\lambda)\lambda}.
\end{equation*}
Our first aim is to estimate the norm of $T$. For $g \in \ell_2(\Lambda)$ we have
\begin{equation*}
 \int{\exp\left(\Re \sum_{\lambda \in \Lambda}{g(\lambda)\lambda}\right)\dd\mu} = \int{\prod_{\lambda \in \Lambda}{\exp(\Re (g(\lambda)\lambda))}\dd\mu}.
\end{equation*}
Since $\exp(ty) \leq \cosh t + y \sinh t$ whenever $t \in \R$ and $-1 \leq y \leq 1$, we have
\begin{equation*}
\int{\exp\left(\Re\sum_{\lambda \in \Lambda}{g(\lambda)\lambda}\right)\dd\mu} \leq \int{\prod_{\lambda \in \Lambda: g(\lambda) \neq 0}{\left(\cosh |g(\lambda)| + \frac{\Re(g(\lambda)\lambda)}{|g(\lambda)|}\sinh |g(\lambda)|\right)}\dd\mu}.
\end{equation*}
Let $\omega:\Lambda \rightarrow \{z \in \C: |z| \leq 1\}$ be such that $\omega(\lambda)|g(\lambda)|\cosh |g(\lambda)|=g(\lambda)\sinh |g(\lambda)|$, so
\begin{equation*}
\int{\exp\left(\Re\sum_{\lambda \in \Lambda}{g(\lambda)\lambda}\right)\dd\mu} \leq \prod_{\lambda \in \Lambda}{\cosh |g(\lambda)|}\int{p_{\omega}\dd\mu} \leq  \exp(1+\|g\|_{\ell_2(\Lambda)}^2/2),
\end{equation*}
since $\cosh x \leq \exp(x^2/2)$ and using that $\Lambda$ is $1$-dissociated w.r.t.\ $\mu$. We conclude that
\begin{equation*}
\frac{1}{(2k)!}\left\|\Re\sum_{\lambda \in \Lambda}{g(\lambda)\lambda}\right\|_{L_{2k}(\mu)}^{2k}\leq  \exp(1+\|g\|_{\ell_2(\Lambda)}^2/2) \text{ for all }g \in \ell_2(\Lambda).
\end{equation*}
Now suppose that $g \in \ell_2(\Lambda)$ has $\|g\|_{\ell_2(\Lambda)}=2\sqrt{k}$. Then
\begin{align*}
\frac{\|Tg\|_{L_{2k}(\mu)}}{\|g\|_{\ell_2(\Lambda)}} & \leq \frac{1}{2\sqrt{k}}(\|\Re Tg\|_{L_{2k}(\mu)} +\|\Re Tig\|_{L_{2k}(\mu)})\\ & \leq \frac{1}{\sqrt{k}}\cdot (2k)!^{1/2k}\exp(K/2k+4k/4k)=O(\sqrt{k}).
\end{align*}
In particular, it follows that  $\|T\|=O(\sqrt{k})$.

The adjoint of $T$ is the map $T^*:L_{\frac{2k}{2k-1}}(\mu) \rightarrow \ell_2(\Lambda); f \mapsto (f\dd\mu)^\wedge|_\Lambda$, and $\|T^*\|=\|T\|=O(\sqrt{k})$. With this we have
\begin{equation*}
|\Lambda| \epsilon^2\mu(A)^2 \leq \sum_{\lambda \in\Lambda}{|(1_A\dd\mu)^\wedge(\lambda)|^2} \leq \|T^*1_A\|_{\ell_2(\Lambda)}^2 = O(k\|1_A\|_{L_{\frac{2k}{2k-1}}(\mu) }^2) =O(k\mu(A)^2\alpha^{-1/2k}).
\end{equation*}
Optimising by taking $k=O(\log2\alpha^{-1})$ and rearranging gives the result.
\end{proof}

\begin{proof}[Proof of Lemma \ref{lem.localchang}]
Let $\refconstbig{llc}>1$ be absolute such that
\begin{equation}\label{eqn.defn}
\refconstbig{llc}\epsilon^{-2}\log 2\alpha^{-1} \geq 4\refconstbig{lcb}\epsilon^{-2}\log 4\alpha^{-1}\text{ for all }\epsilon ,\alpha\in (0,1].
\end{equation}
 Let
\begin{equation*}
\mu:=\mu_{{B_0}+kB_1} \ast \overbrace{\mu_{-B_1} \ast \cdots \ast \mu_{-B_1}}^{k \text{ times}}.
\end{equation*}
By design and hypothesis
\begin{equation*}
\dd\mu_{B_0} \geq 1_{B_0}\dd\mu  = \frac{\mu_G({B_0})}{\mu_G({B_0}+kB_1)}\dd\mu_{B_0} \geq \frac{1}{2}\dd\mu_{B_0},
\end{equation*}
hence writing $A':=A \cap {B_0}$ we have $\alpha \geq \mu(A') \geq \frac{1}{2}\alpha$ and
\begin{equation*}
\Delta:=\left\{\gamma:|(1_A\dd\mu_{B_0})^{\wedge}(\gamma)| \geq \epsilon \alpha\right\} \subset \left\{\gamma : |(1_{A'}\dd\mu)^\wedge(\gamma)| \geq \frac{1}{2}\epsilon \mu(A')\right\}.
\end{equation*}
$\Lambda_0:=\emptyset$ is a subset of $\Delta$ which is $0$-dissociated w.r.t.\ $\mu$. Suppose we have defined $\Lambda_0,\dots,\Lambda_j$ such that $\Lambda_i$ is a subset of $\Delta$ that is $\frac{i}{k+1}$-dissociated w.r.t.\ $\mu$ for all $i \leq j$. If there is some $\gamma \in \Delta\setminus \Lambda_j$ such that $\Lambda_j \cup \{\gamma\}$ is $\frac{j+1}{k+1}$-dissociated then let $\Lambda_{j+1}:=\Lambda_j \cup \{\gamma\}$; otherwise, terminate the iteration. 

Suppose that $\Lambda_j$ is defined for all $j \leq k+1$. Then $\Lambda_{k+1}$ is $1$-dissociated w.r.t.\ $\mu$. Apply Lemma \ref{lem.changbd} with the Lemma's $\Lambda$ equal to $\Lambda_{k+1}$; the Lemma's $A$ equal to $A'$; the Lemma's $\alpha$ equal to $\frac{1}{2}\alpha$; and the Lemma's $\epsilon$ equal to $\frac{1}{2}\epsilon$. We get that  $k+1=|\Lambda_{k+1}| \leq 4\refconstbig{lcb}\epsilon^{-2}\log 4\alpha^{-1}$ which contradicts the hypothesis on $k$ in view of (\ref{eqn.defn}).

In particular, therefore, the iteration terminates for some $j\leq k$; when it terminates set $\Lambda:=\Lambda_j$, and let $B_3$ be the Bohr set with frequency set $\Lambda$ and width $\delta$.

Suppose that $\gamma \in \Delta$.  If $\gamma \in \Lambda$ then $|1-\gamma(x)| \leq \delta$ for all $x \in B_3$, and so (\ref{eqn.estimate}) certainly holds. Hence we may assume that $\gamma \in \Delta \setminus \Lambda$.  Since $\Lambda$ is $\frac{j}{k+1}$-dissociated w.r.t.\ $\mu$ and $\Lambda \cup \{\gamma\}$ is not $\frac{j+1}{k+1}$-dissociated w.r.t.\ $\mu$ (since the iteration terminated), there is some $\omega:\Lambda \rightarrow \{z \in \C: |z|\leq 1\}$ and $w \in \C$ with $|w| \leq 1$ such that
\begin{equation*}
\left|\int{p_\omega \overline{\gamma} \dd\mu}\right| \geq \left|\int{p_{\omega}(1+\Re w \gamma ) \dd\mu}-\int{p_\omega \dd\mu}\right|  >  \exp\left(\frac{j+1}{k+1}\right)- \exp\left(\frac{j}{k+1}\right) \geq \frac{1}{k+1}.
\end{equation*}
Writing $\Span(\Lambda):=\{\sigma_1\lambda_1+\cdots +\sigma_k\lambda_k:\sigma \in \{-1,0,1\}^k\}$, which contains the support of $\wh{p_\omega}$, and applying Plancherel's theorem we get that
\begin{equation*}
 \frac{1}{k+1} \leq \left|\sum_{\lambda \in \Span(\Lambda)}{\wh{p_{\omega}}(\lambda)\wh{\mu}(\gamma-\lambda)}\right|\leq 3^k\sup_{\lambda \in \Span(\Lambda)}{|\wh{\mu_{B_1}}(\gamma-\lambda)|^k}.
\end{equation*}
Hence there is some $\lambda \in \Span(\Lambda)$ such that $|\wh{\mu_{B_1}}(\gamma-\lambda)| \geq 1/6$, and in particular by Lemma \ref{lem.inv}
\begin{equation*}
\gamma-\lambda \in \{\gamma': |1-\gamma'(x)| \leq 12\eta \textrm{ for all }x \in B_2-B_2\}.
\end{equation*}
On the other hand, by the triangle inequality if $\lambda \in \Span(\Lambda)$ then
\begin{equation*}
\lambda \in \{\gamma': |1-\gamma'(x)| \leq k \delta \textrm{ for all }x \in B_3\},
\end{equation*}
and the result follows from a final application of the triangle inequality.
\end{proof}

\section{A local version of the Croot-Sisask almost periodicity lemma}

We need the following local version of the Croot-Sisask almost periodicity lemma from \cite{crosis::}. 
\begin{lemma}\label{lem.crosis} Suppose that $p,L,K \in [2,\infty)$ and $\epsilon \in (0,1]$, $\mu_G(T+{B_0}-{B_0})\leq L \mu_G(T)$, $\mu_G(S+{B_0}) \leq K\mu_G(S)$, and $f\in L_\infty(\mu_{T+{B_0}-{B_0}-S})$. Then there is $t \in {B_0}$ and $X \subset {B_0}-t$ with $\mu_{{B_0}-t}(X) \geq K^{-\newconstbig{csl}\epsilon^{-2}L^{2/p}p}$ and
\begin{equation*}
\|\rho_x(f \ast \mu_{S}) - f \ast \mu_{S}\|_{L_p(\mu_{T})} \leq \epsilon \sup_{s' \in S}{\|f\|_{L_p(\mu_{T+{B_0}-{B_0}-s'})}} \text{ for all }x \in X,
\end{equation*}
where $\rho_x(f)(y)=f(y+x)$.
\end{lemma}
To prove this we need the Marcinkiewicz--Zygmund inequality with the following `$p$-dependence':

\begin{theorem}[Marcinkiewicz--Zygmund inequality, {\cite[Theorem 2, p231]{yaohan::}}]\label{thm.mzi}
Suppose that $p \in [2,\infty)$, $Z_1,\dots,Z_k$ are independent random variables with $\E{Z_i}=0$ and $\E{|Z_i|^p}<\infty$. Then
\begin{equation*}
\E{\left|\sum_{i=1}^k{Z_i}\right|^p} \leq (\newconstbig{mzi}p)^{p/2}k^{p/2-1}\sum_{i=1}^k{\E{|Z_i|^p}}
\end{equation*}
\end{theorem}

\begin{proof}[Proof of Lemma \ref{lem.crosis}]
Let $k:=\lceil 64\refconstbig{mzi}\epsilon^{-2}L^{2/p}p\rceil$ and $\refconstbig{csl}>1$ be absolute such that the inequality $\refconstbig{csl}\epsilon^{-2}L^2p \geq k+1$ holds. Let $s_1,\dots,s_k$ be independent uniformly distributed $S$-valued random variables, and for each $y \in T+{B_0}-{B_0}$ define the random variable $Z_i(y):=f(y-s_i) - f \ast \mu_{S}(y)$.  These are independent and have mean $0$, so it follows by Theorem \ref{thm.mzi} that
\begin{equation*}
\E_{s \in S^k}{\left| \sum_{i=1}^k{Z_i(y)}\right|^p} \leq (\refconstbig{mzi}p)^{p/2}k^{p/2-1}\sum_{i=1}^k{\E_{s \in S^k}{|Z_i(y)|^p}} \text{ for all }y \in T+{B_0}-{B_0}.
\end{equation*}
Integrating against $\dd\mu_{T+{B_0}-{B_0}}(y)$ and dividing by $k^p$ we get
\begin{equation*}
\E_{s \in S^k}{\int{ \left|\frac{1}{k}\sum_{i=1}^k{Z_i(y)}\right|^p\dd\mu_{T+{B_0}-{B_0}}(y)}}\leq (\refconstbig{mzi}pk^{-1})^{p/2}\E_{s \in S^k}{\frac{1}{k}\sum_{i=1}^k{\int{|Z_i(y)|^p\dd\mu_{T+{B_0}-{B_0}}(y)}}}.
\end{equation*}
Now $Z_i(y)=f(y-s_i)-f \ast \mu_{S}(y)$ and so the right hand term can be estimated using the triangle inequality and nesting of norms:
\begin{equation*}
\left(\int{|f(y-s_i) - f \ast \mu_{S}(y)|^p\dd\mu_{T+{B_0}-{B_0}}(y)}\right)^{1/p} \leq 2\sup_{s' \in S}{\|f\|_{L_p(\mu_{T+{B_0}-{B_0}-s'})}}.
\end{equation*}
Hence, recalling the notation $\rho_{-s_i}(f)(y)=f(y-s_i)$ and the definition of $k$, we have
\begin{equation*}
\E_{s \in S^k}{\left\|\frac{1}{k}\sum_{i=1}^k{\rho_{-s_i}(f)} - f \ast \mu_{S}\right\|_{L_p(\mu_{T+{B_0}-{B_0}})}^p}\leq \frac{1}{L}\left(\frac{\epsilon}{4}\right)^p\sup_{s' \in S}{\|f\|_{L_p(\mu_{T+{B_0}-{B_0}-s'})}^p} .
\end{equation*}
Write
\begin{equation*}
\mathcal{L}:=\left\{s \in S^k: \left\|\frac{1}{k}\sum_{i=1}^k{\rho_{-s_i}(f)} - f \ast \mu_{S}\right\|_{L_p(\mu_{T+{B_0}-{B_0}})}\leq \frac{\epsilon}{2L^{1/p}}\sup_{s' \in S}{\|f\|_{L_p(\mu_{T+{B_0}-{B_0}-s'})}}\right\}.
\end{equation*}
By averaging $\P(\mathcal{L}^c) \leq 2^{-p}$ and hence $\P(\mathcal{L}) \geq 1-2^{-p} \geq 1/2$. For $\Delta:=\{(x,\dots,x): x \in {B_0}\}$ we have $\mathcal{L}+\Delta  \subset (S+{B_0})^k$, so $\mu_{G^k}(\mathcal{L} +\Delta)\leq 2K^k\mu_{G^k}(\mathcal{L})$. By the Cauchy-Schwarz inequality
\begin{align*}
 \langle 1_{\Delta} \ast 1_{-\Delta},1_{-\mathcal{L}} \ast 1_{\mathcal{L}}\rangle_{L_2(\mu_{G^k})} & = \|  1_\mathcal{L}\ast 1_\Delta \|_{L_2(\mu_{G^k})}^2\geq \frac{1}{2K^k}\mu_{G^k}(\Delta)^2\mu_{G^k}(\mathcal{L}).
\end{align*}
Hence there is $t \in {B_0}$ and $0_G \in X \subset {B_0}-t$ with $\mu_G(X) \geq \mu_G({B_0})/2K^k\geq K^{-(k+1)}\mu_G({B_0})$ such that for all $x \in X$ then there is $s \in \mathcal{L}$ such that $s-(x,\dots,x) \in \mathcal{L}$, and in particular
\begin{align*}
& \left\|\frac{1}{k}\sum_{i=1}^k{\rho_{-(s_i-x)}(f)} -  \rho_x(f \ast \mu_{S})\right\|_{L_p(\mu_{T})} = \left\|\frac{1}{k}\sum_{i=1}^k{\rho_{-s_i}(f)}- f \ast \mu_{S} \right\|_{L_p(\mu_{T+x})}\\ & \qquad\qquad\qquad\qquad\leq \left(\frac{\mu_G(T+{B_0}-{B_0})}{\mu_G(T)}\right)^{1/p}\frac{\epsilon}{2L^{1/p}}\sup_{s' \in S}{\|f\|_{L_p(\mu_{T+{B_0}-{B_0}-s'})}}.
\end{align*}
Finally, apply the triangle inequality to add the two inequalities when $x=0_G$ and when $x \in X$ is arbitrary.
\end{proof}

\bibliographystyle{halpha}

\bibliography{references}

\begin{thebibliography}{{Sch}16}
\expandafter\ifx\csname url\endcsname\relax
  \def\url#1{\texttt{#1}}\fi
\expandafter\ifx\csname doi\endcsname\relax
  \def\doi#1{\burlalt{doi:#1}{http://dx.doi.org/#1}}\fi
\expandafter\ifx\csname urlprefix\endcsname\relax\def\urlprefix{URL }\fi
\expandafter\ifx\csname href\endcsname\relax
  \def\href#1#2{#2}\fi
\expandafter\ifx\csname burlalt\endcsname\relax
  \def\burlalt#1#2{\href{#2}{#1}}\fi

\bibitem[AH72]{Abbott:1972aa}
H.~Abbott and D.~Hanson.
\newblock A problem of {S}chur and its generalizations.
\newblock {\em Acta Arithmetica}, 20(2):175--187, 1972.
\newblock \urlprefix\url{http://eudml.org/doc/205076}.

\bibitem[Bou99]{bou::5}
J.~Bourgain.
\newblock On triples in arithmetic progression.
\newblock {\em Geom. Funct. Anal.}, 9(5):968--984, 1999.
\newblock \doi{10.1007/s000390050105}.

\bibitem[BS23]{Bloom:2023aa}
T.~F. Bloom and O.~Sisask.
\newblock The {K}elley--{M}eka bounds for sets free of three-term arithmetic
  progressions.
\newblock {\em Essential Number Theory}, 2(1):15--44, 2023.
\newblock \doi{10.2140/ent.2023.2.15}.

\bibitem[Cha02]{cha::0}
M.-C. Chang.
\newblock A polynomial bound in {F}reiman's theorem.
\newblock {\em Duke Math. J.}, 113(3):399--419, 2002.
\newblock \doi{10.1215/S0012-7094-02-11331-3}.

\bibitem[CP20]{chapre::}
J.~Chapman and S.~Prendiville.
\newblock On the {R}amsey number of the {B}rauer configuration.
\newblock {\em Bulletin of the London Mathematical Society}, 52(2):316--334,
  2020, \burlalt{arXiv:1904.07567}{http://arxiv.org/abs/arXiv:1904.07567}.
\newblock \doi{10.1112/blms.12327}.

\bibitem[CS10]{crosis::}
E.~S. Croot and O.~Sisask.
\newblock A probabilistic technique for finding almost-periods of convolutions.
\newblock {\em Geom. Funct. Anal.}, 20(6):1367--1396, 2010,
  \burlalt{arXiv:1003.2978}{http://arxiv.org/abs/arXiv:1003.2978}.
\newblock \doi{10.1007/s00039-010-0101-8}.

\bibitem[CS17]{cwasch::0}
K.~Cwalina and T.~Schoen.
\newblock Tight bounds on additive {R}amsey-type numbers.
\newblock {\em Journal of the London Mathematical Society}, 96(3):601--620,
  2017.
\newblock \doi{10.1112/jlms.12081}.

\bibitem[FK06]{foxkle::}
J.~Fox and D.~J. Kleitman.
\newblock On {R}ado's boundedness conjecture.
\newblock {\em J. Combin. Theory Ser. A}, 113(1):84--100, 2006.
\newblock \doi{10.1016/j.jcta.2005.07.004}.

\bibitem[GG55]{gregle::0}
R.~E. {Greenwood} and A.~M. {Gleason}.
\newblock {Combinatorial relations and chromatic graphs.}
\newblock {\em {Can. J. Math.}}, 7:1--7, 1955.
\newblock \doi{10.4153/CJM-1955-001-4}.

\bibitem[Gow98]{gow::4}
W.~T. Gowers.
\newblock A new proof of {S}zemer\'edi's theorem for arithmetic progressions of
  length four.
\newblock {\em Geom. Funct. Anal.}, 8(3):529--551, 1998.
\newblock \doi{10.1007/s000390050065}.

\bibitem[GR07]{greruz::0}
B.~J. Green and I.~Z. Ruzsa.
\newblock Freiman's theorem in an arbitrary {A}belian group.
\newblock {\em J. Lond. Math. Soc. (2)}, 75(1):163--175, 2007,
  \burlalt{arXiv:math/0505198}{http://arxiv.org/abs/arXiv:math/0505198}.
\newblock \doi{10.1112/jlms/jdl021}.

\bibitem[Gre04]{gre::6}
B.~J. Green.
\newblock Spectral structure of sets of integers.
\newblock In {\em Fourier analysis and convexity}, Appl. Numer. Harmon. Anal.,
  pages 83--96. Birkh\"auser Boston, Boston, MA, 2004.
\newblock \doi{10.1007/978-0-8176-8172-2\_4}.

\bibitem[Gre05]{gre::9}
B.~J. Green.
\newblock Finite field models in additive combinatorics.
\newblock In {\em Surveys in combinatorics 2005}, volume 327 of {\em London
  Math. Soc. Lecture Note Ser.}, pages 1--27. Cambridge Univ. Press, Cambridge,
  2005, \burlalt{arXiv:math/0409420}{http://arxiv.org/abs/arXiv:math/0409420}.
\newblock \doi{10.1017/CBO9780511734885.002}.

\bibitem[HW60]{Hardy:1960aa}
G.~H. Hardy and E.~M. Wright.
\newblock {\em An Introduction to the Theory of Numbers}.
\newblock Clarendon Press, Oxford, 4th edition, 1960.

\bibitem[KM23]{kelmek::0}
Z.~Kelley and R.~Meka.
\newblock Strong bounds for $3$-progressions, 2023.
\newblock \doi{10.48550/ARXIV.2302.05537}.

\bibitem[Ko{\'s}25]{Kosciuszko:2025aa}
T.~Ko{\'s}ciuszko.
\newblock {S}chur-like numbers and a lemma of {S}hearer, 2025,
  \burlalt{2507.21656}{http://arxiv.org/abs/2507.21656}.
\newblock \urlprefix\url{https://arxiv.org/abs/2507.21656}.

\bibitem[LR14]{lanrob::}
B.~M. Landman and A.~Robertson.
\newblock {\em Ramsey theory on the integers}, volume~73 of {\em Student
  Mathematical Library}.
\newblock American Mathematical Society, Providence, RI, second edition, 2014.
\newblock \doi{10.1090/stml/073}.

\bibitem[Pel24]{Peluse:2024aa}
S.~Peluse.
\newblock {\em Finite Field Models in Arithmetic Combinatorics -- Twenty Years
  On}, pages 159--200.
\newblock London Mathematical Society Lecture Note Series. Cambridge University
  Press, 2024.
\newblock \doi{10.1017/9781009490559.007}.

\bibitem[Rad33]{rad::1}
R.~Rado.
\newblock Studien zur {K}ombinatorik.
\newblock {\em Mathematische Zeitschrift}, 36:424--480, 1933.
\newblock \urlprefix\url{http://eudml.org/doc/168408}.

\bibitem[RL01]{yaohan::}
Y.-F. Ren and H.-Y. Liang.
\newblock On the best constant in {M}arcinkiewicz-{Z}ygmund inequality.
\newblock {\em Statistics \& Probability Letters}, 53(3):227--233, 2001.
\newblock \doi{10.1016/S0167-7152(01)00015-3}.

\bibitem[{Sch}16]{sch::4}
I.~{Schur}.
\newblock {\"Uber die Kongruenz $x^m+y^m\equiv z^m (\text{mod.} p)$.}
\newblock {\em {Jahresber. Dtsch. Math.-Ver.}}, 25:114--117, 1916.
\newblock \urlprefix\url{http://eudml.org/doc/145475}.

\bibitem[TV06]{taovu::}
T.~C. Tao and V.~H. Vu.
\newblock {\em Additive combinatorics}, volume 105 of {\em Cambridge Studies in
  Advanced Mathematics}.
\newblock Cambridge University Press, Cambridge, 2006.
\newblock \doi{10.1017/CBO9780511755149}.

\bibitem[Wol15]{wol::3}
J.~Wolf.
\newblock Finite field models in arithmetic combinatorics---ten years on.
\newblock {\em Finite Fields Appl.}, 32:233--274, 2015.
\newblock \doi{10.1016/j.ffa.2014.11.003}.

\end{thebibliography}

\end{document}